\documentclass{article}
\usepackage{amsmath}
\usepackage{amssymb}
\usepackage{amsmath}
\usepackage{amsfonts}
\usepackage{graphicx}
\usepackage[latin1]{inputenc}
\usepackage[english]{babel}
\usepackage{dsfont,hyperref}
\usepackage{stackrel,stmaryrd}
\DeclareMathOperator*{\esssup}{ess\,sup}
\DeclareMathOperator*{\essinf}{ess\,inf}
\usepackage{url}

\newenvironment{class}[1][(2020) Mathematics Subject Classification]{\textbf{#1.} }{}
\newenvironment{MC}[1][Key words and phrases]{\textbf{#1.} }{}
\usepackage{vmargin}
\setmarginsrb{2cm}{2cm}{2cm}{1cm}{0cm}{0cm}{0cm}{0cm}
\providecommand{\U}[1]{\protect\rule{.1in}{.1in}}
\newtheorem{theorem}{Theorem}

\newtheorem{lemma}[theorem]{Lemma}

\newtheorem{proposition}[theorem]{Proposition}
\newtheorem{remark}[theorem]{Remark}

\newenvironment{proof}[1][Proof]{\noindent\textbf{#1.} }{\ \rule{0.5em}{0.5em}}

\begin{document}
	\title{Generalized Reflected BSDEs with RCLL Random Obstacles in a General Filtration\thanks{This work constitutes the main paper on which the results of the recent paper \cite{ElmansouriROSE2025} are based, and from which the main results of that paper are derived.}}
	\author{Badr ELMANSOURI\footnote{Corresponding author} \\
		Cadi Ayyad University (UCA), National School of Applied Sciences of Marrakech (ENSA-M),\\ BP 575, Avenue Abdelkrim Khattabi, 40000, Guéliz, Marrakech, Morocco.\\
		Emails: \url{b.elmansouri@uca.ac.ma} \& \url{badr.elmansouri@edu.uiz.ac.ma} \\\\
		 Mohamed EL OTMANI\\
		Faculty of Sciences Agadir, Ibn Zohr University,\\ Laboratory of Analysis and Applied Mathematics (LAMA),\\
		Hay Dakhla, BP8106, Agadir, Morocco\\
		Emails: \url{m.elotmani@uiz.ac.ma}}

	\maketitle
\begin{abstract}
	This paper addresses the existence and uniqueness of solutions to Reflected Generalized Backward Stochastic Differential Equations (GRBSDEs) within a general filtration that supports a Brownian motion and an independent integer-valued random measure. Our study focuses on cases where the given data satisfy appropriate $\mathbb{L}^2$-integrability conditions and the coefficients satisfy a monotonicity assumption. Additionally, we establish a connection between the solution and an optimal control problem over the set of stopping times.
\end{abstract}

	\begin{MC}
		Reflected Generalized BSDEs, Integer-valued random jump measure, General filtration, Monotonic coefficients.
	\end{MC}
	
	\begin{class}
		60H05; 60H20; 60H30; 60G46, 60G57.
	\end{class}

	\section{Introduction} \label{Intro}
The use of probabilistic methods to study partial differential equations (PDEs) is a classical approach, particularly since the establishment of the connection between solutions of linear PDEs and diffusion processes via the celebrated Feynman-Kac formula. In 1990, Pardoux and Peng \cite{pardouxpengadapted} introduced a new class of stochastic differential equations, known as backward stochastic differential equations (BSDEs). For a brief mathematical description, let $T > 0$ be a finite time horizon, and let $(\Omega, \mathcal{F}, \mathbb{P})$ be a probability space on which a $d$-dimensional Brownian motion $(B_t)_{t \leq T}$ is defined. We denote by $\mathbb{F} := (\mathcal{F}_t)_{t \in [0,T]}$ the filtration generated by the Brownian motion $W$. A backward SDE is an equation of the form
\begin{equation}\label{eq.1}
	Y_t = \xi + \int_{t}^{T} f(s, Y_s, Z_s)\, ds - \int_{t}^{T} Z_s\, dB_s, \quad t \in [0,T].
\end{equation}
The terminal condition $\xi$ is an $\mathcal{F}_T$-measurable random variable taking values in $\mathbb{R}^k$, and the generator $f$ (or coefficient) is a stochastic function defined on $[0, T] \times \Omega \times \mathbb{R}^k \times \mathbb{R}^{k \times d}$, taking values in $\mathbb{R}^k$, and measurable with respect to the $\sigma$-algebra $\mathcal{P} \otimes \mathcal{B}(\mathbb{R}^k) \otimes \mathcal{B}(\mathbb{R}^{k \times d})$, where $\mathcal{P}$ denotes the predictable $\sigma$-field on $\Omega \times [0,T]$.

Solving such an equation consists in finding a pair of processes $(Y_t, Z_t)_{t \leq T}$ that are $\mathbb{F}$-adapted and satisfy equation \eqref{eq.1}.

Subsequently, Peng \cite{PengPDE} provided a probabilistic interpretation of the solution to a semilinear PDE of parabolic or elliptic type, thereby generalizing the Feynman-Kac formula to nonlinear PDEs. Since then, the theory of BSDEs has developed rapidly due to its strong connections with various fields (see also \cite{ePardouxB}).

The introduction of an additional integral term, given by a Stieltjes integral with respect to a continuous and increasing process in the formulation of the BSDE \eqref{eq.1}, leads to a new class of backward stochastic differential equations, referred to as \textit{generalized} BSDEs (GBSDEs). This integration process represents the \textit{local time} of a diffusion process at the boundary. The aim of this extension is to provide a probabilistic representation of the solutions to systems of semilinear partial differential equations, either parabolic or elliptic, subject to nonlinear Neumann-type boundary conditions. This line of research was initiated by Pardoux and Zhang in \cite{Pardoux Z}, where, in the same probabilistic framework as that considered for the classical BSDE \eqref{eq.1}, the generalized formulation takes the following form:
\begin{equation}\label{eq.2}
	Y_t = \xi + \int_t^T f(s, Y_s, Z_s)\, ds + \int_t^T g(s, Y_s)\, dA_s - \int_t^T Z_s\, dB_s, \quad t \in [0,T],
\end{equation}
where $(A_t)_{t \in [0,T]}$ is an increasing process associated with the boundary local time, and $(Y, Z)$ denotes the solution of the generalized BSDE.

	In their paper, the authors establish the existence and uniqueness of the solution under monotonicity assumptions on the generators $f$ and $g$ with respect to the variable $y$, a Lipschitz condition on $f$ with respect to $z$, as well as appropriate linear growth and square integrability conditions. Moreover, they place the GBSDE \eqref{eq.2} in a Markovian framework, where the randomness is driven by a reflected diffusion process described by the following stochastic differential equation:
	\begin{equation*}
		\left\lbrace 
		\begin{aligned}
			X^x_t &= x + \int_0^t b(X^x_s)\, ds + \int_0^t \sigma(X^x_s)\, dB_s + \int_0^t \nabla \Phi(X^x_s)\, dA^x_s, \quad t \in [0,T],\\
			dA^x_t &= \mathds{1}_{\{X^x_t \in \partial G\}}\, dA^x_t,
		\end{aligned}
		\right.
	\end{equation*}
	where $\Phi$ is a function of class $\mathcal{C}^2_b(\mathbb{R}^l)$, $b$ and $\sigma$ are Lipschitz continuous deterministic functions, $G$ is a bounded open subset of $\mathbb{R}^l$, and $\partial G$ denotes its boundary, defined via $\Phi$ by $G = \{\Phi > 0\}$ and $\partial G = \{\Phi = 0\}$. For any $x \in \partial G$, the vector $\nabla \Phi(x)$ is assumed to point inward, defining the inward normal derivative:
	$$
	\frac{\partial}{\partial n} = \langle \nabla \Phi(x), D_x \cdot \rangle = \sum_{i=1}^l \frac{\partial \Phi}{\partial x_i}(x) \frac{\partial}{\partial x_i}.
	$$
	
	The authors then show that, under this Markovian framework, the process $(Y_t)_{t \leq T}$ solving the GBSDE \eqref{eq.2} provides a probabilistic representation of the solution to the following partial differential equation with a Neumann-type boundary condition:
	\begin{equation*}
		\left\{
		\begin{aligned}
			-\frac{\partial u}{\partial t}(t,x) - {\mathcal{L}} u(t,x) - f(t,x,u(t,x), (\nabla u \, \sigma^\top)(t,x)) &= 0, \quad &\text{for } (t,x) \in [0,T] \times G,\\
			u(T,x) &= H(x), \quad &\text{for } x \in G,\\
			\frac{\partial u}{\partial n}(t,x) + g(t,x,u(t,x)) &= 0, \quad &\text{for } x \in \partial G,
		\end{aligned}
		\right.
	\end{equation*}
	where $\mathcal{L}$ is the second-order partial differential operator
	$$
	\mathcal{L} = \frac{1}{2} \sum_{i, j=1}^d (\sigma \sigma^*)_{ij}(x) \frac{\partial^2}{\partial x_i \partial x_j} + \sum_{i=1}^d b_i(x) \frac{\partial}{\partial x_i}.
	$$
	
	The extension of discontinuous generalized BSDEs to a filtration generated by both a Brownian motion $(B_t)_{t \leq T}$ and a Poisson random measure $N$ on $[0,T] \times E$, with $E = \mathbb{R}^{\ell} \setminus \{0\}$ for some strictly positive integer $\ell$, was studied by Pardoux in \cite{Pardoux}. In this work, the author proves that the following equation:
	\begin{equation}\label{eq.3}
		Y_t = \xi + \int_t^T f(s, Y_s, Z_s, V_s)\, ds + \int_t^T g(s, Y_s)\, dA_s - \int_t^T Z_s\, dB_s - \int_t^T \int_E V_s(e)\, \tilde{N}(ds,de), \quad t \in [0,T],
	\end{equation}
	admits a unique solution $(Y,Z,U)$ under assumptions similar to those in \cite{Pardoux Z}, namely monotonicity and appropriate growth conditions on $f$ and $g$, as well as a Lipschitz condition on $f$ with respect to the variable $u$.
	
	Furthermore, GBSDEs driven by Lévy processes were considered by El Otmani in \cite{ELOTMANI G}, where the author established a theoretical framework guaranteeing the existence and uniqueness of the solution. In addition, an explicit connection was made with a class of integro-partial differential equations subject to Neumann boundary conditions, thereby extending the applicability of GBSDEs to settings involving discontinuous dynamics induced by Lévy jumps.
	
	In a more general filtration setting, Elmansouri and El Otmani \cite{ELMandELO} studied GBSDEs where the filtration \textit{supports} a Brownian motion and an independent integer-valued random measure. Specifically, the authors established existence and uniqueness results for the following BSDE:
	\begin{equation}\label{eq.4}
		Y_t = \xi + \int_t^T f(s, Y_s, Z_s, V_s)\, ds + \int_t^T g(s, Y_s)\, dA_s - \int_t^T Z_s\, dB_s - \int_t^T \int_E V_s(e)\, \tilde{N}(ds,de) - \int_{t}^{T} dM_s,~t \in [0,T],
	\end{equation}
	where $(M_t)_{t \leq T}$ is a square-integrable martingale orthogonal to both $W$ and $N$, under the same assumptions on the drivers as in \cite{ELOTMANI G, Pardoux, Pardoux Z}, thus generalizing and improving those results. This result was further extended in a more general setting, where the noise is driven by a general RCLL martingale, by the same authors in \cite{ELMandELOVMSTA}.
	
	It is worth noting that in the proofs of the existence results for the GBSDEs \eqref{eq.2} and \eqref{eq.3}, the main tool is the martingale representation theorem, which holds in both cases: when the filtration is generated by a Brownian motion alone or jointly with an independent integer-valued random measure. However, for more general filtrations, this property does not hold, and an additional martingale term must be included in the definition of the solution. Therefore, in the GBSDE \eqref{eq.4}, the solution is a quadruplet $(Y, Z, V, M)$ rather than the classical triplet used in \eqref{eq.3}. This additional stochastic component introduces further complexity to the analysis, as it must also be properly controlled.

	Generalized reflected BSDEs (GRBSDEs, for short) with one continuous barrier were first studied by Ren and Xia \cite{Ren and Xia} in the Brownian setting following the work by El Karoui et al. \cite{ElKaroui}, and later extended to the case with Poisson jumps by Elhachemy and El Otmani \cite{HACHEL}, motivated by obstacle problems for integro-PDEs with nonlinear Neumann boundary conditions. Ren and El Otmani \cite{Ren El Otmani} further investigated GRBSDEs driven by a Lévy process within the same PDE framework. In the case of a filtration generated by a Brownian motion and an independent Poisson random measure, such equations take the following form:
	\begin{equation}\label{eq.5}
		\left\{
		\begin{split}
			\text{ } &~Y_t= \xi+\int_t^T f(s,Y_s,Z_s,V_s)\,ds+\int_t^T g(s,Y_s)\,\mathrm{d}A_s+K_T-K_t-\int_t^T Z_s\, d B_s \\
			&\quad -\int_t^T \int_E V_s(e)\,\tilde{N}(ds,de),\quad 0 \leq t \leq T. \\
			\text{ } &~ Y_t \geq L_t,\ t\leq T ~~ \text{ and } ~~ \int_0^T (Y_{t-}-L_{t-})\,dK_t = 0.
		\end{split}
		\right.
	\end{equation}
	
	The Skorokhod condition in equation \eqref{eq.5} ensures that the first component $Y$ of the solution remains above the barrier $L$. It is worth noting that, in this setting, the jumps in the solution may be either inaccessible arising from the stochastic integral with respect to $N$ or predictable, resulting from the predictable negative jumps of the barrier process $L$. The process $K$ acts as a minimal non-decreasing compensator that pushes $Y$ upwards when it tends to fall below $L$, as specified by the minimality condition: the continuous part $K^c$ of $K$ increases only when $Y = L$, and the jump part $K^d$ is activated on the random set $\{Y_{-} = L_{-}\} \cap \{\Delta L < 0\}$ to compensate for negative predictable jumps in $L$.
	
	To the best of our knowledge, there are no existing results on the existence and uniqueness of such GRBSDEs under a general filtration that supports both a Brownian motion and an independent integer-valued random measure, satisfying the monotonicity assumptions on the generators as stated in \cite{ELOTMANI G, Pardoux, Pardoux Z} for classical GBSDEs and in \cite{HACHEL} for the reflected case. Establishing well-posedness results for this type of reflected GBSDE would enable us to tackle more general problems, such as the stochastic monotonic case, and to prove existence and uniqueness results for doubly reflected GBSDEs by analyzing the convergence of penalized schemes. These penalizations both increasing and decreasing rely on the notion of local solutions, following the arguments developed by Elmansouri \cite{BadrPJM} and Hamadène and Wang \cite{hamadenewang}. This result can then be used to treat GBSDEs with stochastic monotonic generators in this general framework.
	
	Since the construction of such local solutions for doubly reflected BSDEs naturally leads to the study of a reflected generalized BSDE in a general filtration setting and given the absence of results on the existence and uniqueness of GRBSDEs in this context, we are motivated to address this gap. In this paper, we study GRBSDEs in a general filtration, providing existence and uniqueness results under monotonicity assumptions on the generators and in the presence of a reflecting obstacle that may exhibit both predictable and totally inaccessible jumps. Our method is based on a penalization approach, combined with the Snell envelope and tools from optimal stopping theory. We also derive several useful a priori estimates for the solution and establish a link between this type of GRBSDE and an associated optimal stopping problem, characterizing the solution process as the value function of this problem. Furthermore, we present comparison principles that play a key role in analyzing the convergence of penalized schemes for GRBSDEs toward the solution of the corresponding doubly reflected GBSDE. For further related and interesting contributions, we refer the reader to \cite{CRP,Cvitanic,HACHEL2,AMR,StochDyn,OtmaniR,El Otmani and Mohamed,Essaky,ESAS,FakhouriOR,HamadeneoneR,HamadeneHass,HamadeneO,HOKN}, among others.
	
	The rest of the paper is organized as follows: In Section~\ref{PRLM}, we present the preliminary notions and assumptions used throughout the paper. Section~\ref{Generalized BSDEs with one reflecting barrier : Existence and uniqueness result} is devoted to the existence and uniqueness results for GRBSDEs with a single reflecting barrier, allowing jumps of general structure, under monotonicity conditions on the coefficients. A comparison principle is also established. In Section~\ref{app}, we provide an application of this type of BSDE in the context of an optimal control problem over the set of stopping times.

	\section{Preliminaries}
	\label{PRLM}
	All processes are defined on a probability space $\left( \Omega,\mathcal{F},\mathbb{P} \right)$ equipped with a filtration $\mathbb{F}=\left( \mathcal{F}_t \right)_{t \geq 0}$ carrying a $d$-dimensional Brownian motion $\{ B_t , t \geq 0 \}$ and an independent random measure $N$ on $\mathbb{R}^{+} \times E$. The filtration $\mathbb{F}$ is assumed to be complete, right continuous, and quasi-left continuous, meaning that for every sequence $\left\lbrace \tau_n\right\rbrace_{n \in \mathbb{N}}$ of $\mathbb{F}$-stopping times such that $\tau_n \nearrow {\tau}$ a.s. for some stopping time ${\tau}$, we have $\bigvee_{n \in \mathbb{N}} \mathcal{F}_{\tau_n} = \mathcal{F}_{\tau}$. We set $E=\mathbb{R}^{\ell}\setminus\{0\}$ for some $\ell \in \mathbb{N}^{+}$, equipped with its Borel field $\mathcal{E}$. We always assume that the random measure $N$ is an integer-valued random measure on $(\mathbb{R}_{+} \times E,\mathcal{B}(\mathbb{R}^{+}) \otimes \mathcal{E})$ with the compensator
	$\upsilon(\omega; dt  , de )=Q(\omega,t; de )\eta(\omega,t) dt  $ where  $\eta : \Omega \times \mathbb{R}_{+} \rightarrow [0,\infty)$ is a predictable process and $Q$ is a kernel from $\left(\Omega \times \mathbb{R}_{+}, \mathcal{P}  \right)$ into $(E,\mathcal{E})$ satisfying $\int_0^t \int_E \left| e \right|^2 Q(t, de )\eta(t) dt  <\infty$. We also set $N (\{0\},E)=N((0,\infty),\{0\})=\upsilon((0,\infty),\{0\})=0$. The quadratic variation of a martingale $M$ is defined by $[M]$. The notation $\left\langle M^c\right\rangle =[M]^c$ will denote the continuous part of the quadratic variation $[M]$. The Euclidean norm of a vector $z \in \mathbb{R}^k$ will be defined by $\left\| z \right\|^2=\sum_{i=1}^k \left| z_i \right|^2$.\\
	Let $\eta_1$ and $\eta_2$ by two $[0,T]$-valued $\mathbb{F}$-stopping times such that $\eta_1 \leq \eta_2$ a.s. By $\mathcal{T}_{\eta_1}^{\eta_2}$, we denote the set of $\mathbb{F}$-stopping times $\eta$ such that $\eta_1 \leq \eta \leq \eta_2$ a.s. The equality $X=Y$ between any two processes $(X_t)_{t \geq 0}$ and $(Y_t)_{t \geq 0}$ must be understood in the indistinguishable sense, meaning that $\mathbb{P}\left(\omega:X_t(\omega)=Y_t(\omega), \forall t \geq 0\right)=1$. The same interpretation applies to $X \leq Y$. While these relations are understood in the almost sure sense for random variables. For a given RCLL process $(Y_t)_{t \geq 0}$, $Y_{t-}=\lim\limits_{s \nearrow t } Y_s$ denotes the left limit of $Y$ at $t$. We set $Y_{0-}=Y_0$ by convention, and $Y_{-}=(Y_{t-})_{t \leq T}$ represents the left-limited process. Moreover, $\Delta Y_t=Y_t-Y_{t-}$ denotes the jump of $Y$ at time $t \geq 0$. For $\mathbb{F}$-progressively measurable RCLL processes $\left\lbrace Y^n\right\rbrace _{n \in \mathbb{N}}$ and $Y$, we say that $Y^n \rightarrow Y$ in \textsc{UCP} (uniformly on compacts in probability) if $\sup_{0 \leq s \leq t} \left| Y^n_s -Y_s \right|^2 \rightarrow 0$ in probability $\mathbb{P}$ for every $t >0$. Additionally, for $x \in \mathbb{R}$, we have $x^{+}=\max(x,0)$ and $x^{-}=\max(-x,0)=-\min(x,0)$.
	%To simplify the notation, we generally omit any dependence on $\omega$ of a given process or random function, and by convention, all brackets and stochastic integrals are taken null at time zero.
	
	Let $\tilde{\mathcal{P}} = \mathcal{P} \otimes \mathcal{E}$ be the $\sigma$-field on $\Omega \times \mathbb{R}_{+} \times E$, where $\mathcal{P}$ designates the $\sigma$-algebra on $\Omega \times [0,T]$ generated by all $\mathbb{F}$-adapted left-continuous processes. Let's denote:
	\begin{itemize}
		\item $G_{\text{loc}}(N)$: the set of  $\tilde{\mathcal{P}}$-measurable functions $V$ on $\Omega \times \mathbb{R}_{+}\times E$ such that for any $t \leq T$ a.s.
		\begin{equation*}
			\int_0^t \int_E \left( \left| V_s(e) \right|^2 \wedge \left| V_s(e) \right|  \right) Q(s, de )\eta(s) ds < \infty.
		\end{equation*}
		\item $L^2_{\text{loc}}(B)$: the set of all $\mathcal{P}$-measurable processes $Z$ such that a.s.
		$\int_0^T \| Z_s \|^2  ds <\infty$.
		\item $\mathbb{M}_{\text{loc}}$: the set of RCLL local martingales orthogonal to $B$ and $\tilde{N}$. If $M \in \mathbb{M}_{loc}$ then
		$\left[ M,W^i \right]^c_t=0, ~ 1 \leq i \leq d$ and $\left[M,\tilde{N}(\cdot,\mathcal{A})  \right]_t=0$, for all $\mathcal{A} \in \mathcal{E}$.
		\item $\mathbb{L}^2_Q=L^2(E;\mathbb{R}^k)$: the set of measurable function $\phi : E \rightarrow \mathbb{R}^k$ such that 
		\begin{equation*}
			\left\| \phi \right\|^2_{Q} =  \int_{E} \left| \phi(e) \right|^2 Q(t, de )\eta(t)< \infty, ~ \mbox{ for almost every } (t,\omega)\in [0,T] \times \Omega.
		\end{equation*}
	\end{itemize}
	
	For a fixed deterministic horizon time $T > 0$, a parameter $\mu > 0$, a Hilbert space $H$, and a continuous increasing $\mathbb{F}$-adapted process $(A_t)_{0 \leq t \leq T}$, we introduce the following spaces of processes:
	\begin{itemize}
		\item[{$\bullet$}] $M^2_{\mu}(A;H)$: the set of $H$-valued progressively measurable processes $X:=\left\lbrace X_t ; 0 \leq t \leq T  \right\rbrace$, which are such that
		\begin{equation*}
			\mathbb{E}\left[ \int_0^T e^{\mu A_t} \| X_t \|^2_H  dt  +\int_0^T e^{\mu A_t} \| X_t \|^2_H  dA  _t  \right]<\infty.
		\end{equation*}
		\item[{$\bullet$} ]$M^{2,A}_{\mu}(H)$: the same space $M^2_{\mu}(A;H)$  without the first term, i.e the same space, where we require only that
		\begin{equation*}
			\| X  \|_{M^{2,A}_{\mu}(H)}^2=\mathbb{E}\left[ \int_0^T e^{\mu A_t} \| X_t \|^2_H  dA_t    \right]<\infty
		\end{equation*}
		(with the convention $M^{2,A} :=M^{2,A}_{0}$).
		\item[{$\bullet$} ]$M^2_{\mu}(H)$ : the same space $M^2_{\mu}(A;H)$  without the second term, i.e the same space, where we require only that
		\begin{equation*}
			\| X  \|_{M^2_{\mu}(H)}^2=\mathbb{E}\left[ \int_0^T e^{\mu A_t} \| X_t \|^2_H  dt    \right]<\infty
		\end{equation*}
		(with the convention $M^2(H):=M^2_{0}(H)$).
		\item[{$\bullet$}] $\mathbb{M}^2_{\mu}$: the subspace of $\mathbb{M}_{loc}$ of  martingales $M:=\left\lbrace M_t ; 0 \leq t \leq T  \right\rbrace$  such that
		\begin{equation*}
			\| M  \|_{\mathbb{M}^2_{\mu}}^2=\mathbb{E}\left[ \int_0^T e^{\mu A_t}  d[M]_t \right] < \infty
		\end{equation*}
		(with the convention  $\mathbb{M}^2:=\mathbb{M}^2_{0}$).
		\item[{$\bullet$}] $\mathcal{S}^2_{\mu}(A;\mathbb{R})$: the set of $\mathbb{R}$-valued, $\mathbb{F}$-adapted RCLL processes $(Y_t)_{0 \leq t \leq T}$ such that
		$$
		\| Y \|^2_{\mathcal{S}^2_{\mu}(A;\mathbb{R})} =\mathbb{E}\left[ \sup_{0 \leq s \leq T} e^{\mu  A_s} \left| Y_s \right|^2
		+\int_0^T e^{\mu  A_s} \left| Y_s \right|^2  dA_s \right]<\infty.
		$$
		\item[{$\bullet$}] $\mathcal{S}^2_{\mu}$: the same space of $\mathcal{S}^2_{\mu}(A;\mathbb{R})$ without the second term, i.e the same space, where we require only that
		$$
		\| Y \|^2_{\mathcal{S}^2_{\mu}}=\mathbb{E}\left[ \sup_{0 \leq s \leq T} e^{\mu  A_s} \left| Y_s \right|^2 \right]<\infty.
		$$
		(with the convention  $\mathcal{D}^2:=\mathcal{S}^2_{0}$).
		\item[{$\bullet$}] $\mathcal{S}^2$: the set of $\mathbb{R}$-valued, $\mathbb{F}$-predictable RCLL  increasing processes $(K_t)_{0 \leq t \leq T}$ such that $K_0=0$ and $\mathbb{E}\left[  \left| K_T \right|^2 \right]<\infty$.
		\item[{$\bullet$}] $L^2_{\text{loc},\mu}(B)$: the set of $Z \in G_{\text{loc}}(N)$ such that a.s.
		$\int_0^T e^{\mu A_s}\left\| Z_s \right\|^2  ds <\infty$.
		\item[{$\bullet$}] $G^2_{\text{loc},\mu}(N)$: the set of $V \in G_{\text{loc}}(N)$ such that a.s.
		$\int_0^T e^{\mu A_s}\int_E \left| V_s(e) \right|^2_Q  ds <\infty$.
		\item[{$\bullet$}] $\mathbb{M}^2_{\text{loc},\mu}$: the set of $M \in \mathbb{M}_{\text{loc}}$ such that a.s.
		$\int_0^T e^{\mu A_s} d\left[M\right]_s <\infty$.
	\end{itemize}
	
	Since we are dealing with a general filtration $\mathbb{F}$, we recall the following Lemma which gives the representation property of a local martingale:
	\begin{lemma}[Lemma III.4.24 in \cite{jacodshiryaev}]
		Every $\mathbb{F}$-local martingale has a decomposition
		\begin{equation*}
			\int_0^{\cdot} Z_s  dB   _s +\int_0^{\cdot} \int_E V_s(e)\tilde{N}( ds , de )+M,
		\end{equation*}
		where $M \in \mathbb{M}_{loc}$, $Z \in L^2_{loc}(B)$ and $V \in G_{loc}(N)$.
		\label{Lemma for the representation in general}
	\end{lemma}
	\paragraph*{Assumptions on the data $(\xi,f,g,A,L,U)$}
	We are giving:
	\begin{itemize}
		\item[\textbf{(H1) }] A final condition  $\xi$ which is an $\mathcal{F}_T$-measurable random variable that satisfies
		\begin{equation*}
			\| \xi \|^2_{\mathbb{L}^2_{\mu}}=\mathbb{E}\left[ e^{\mu A_T}\left| \xi \right|^2 \right] < \infty.
		\end{equation*}
		
		\item[\textbf{(H2) }] Two coefficients $f : \Omega \times [0,T] \times \mathbb{R} \times \mathbb{R}^d \times \mathbb{L}^2_Q \rightarrow \mathbb{R}$ and $g : \Omega \times [0,T] \times \mathbb{R} \rightarrow \mathbb{R}$, satisfying for some constants $\alpha \in \mathbb{R}$, $\beta < 0$, $\kappa >0$, some $[1,\infty)$-valued $\mathbb{F}$-adapted processes $\{ \varphi_t,\psi_t, 0 \leq t \leq T \}$ and all $(t,y,z,v),(t,y^{\prime},z^{\prime},v^{\prime})\ \in [0,T] \times \mathbb{R}^{1+d} \times \mathbb{L}^2_Q$,
		\begin{itemize}
			\item[(i) ] $f(\cdot,y,z,v)$ and $g(\cdot,y)$ are two $\mathbb{F}$-progressively measurable processes,
			\item[(ii) ] $\mathbb{E}\left[  \displaystyle\int_0^T e^{\mu A_t} \varphi^2_t dt + \displaystyle\int_0^T e^{\mu A_t} \psi^2_t dA_t     \right]  <\infty $,
			\item[(iii) ] $\left( y-y^{\prime} \right)\left( f(t,y,z,v)-f(t,y^{\prime},z,v) \right) \leq \alpha \left| y-y^{\prime} \right|^2$,
			\item[(iv) ] $\left( y-y^{\prime} \right) \left( g(t,y)-g(t,y^{\prime}) \right) \leq \beta \left| y-y^{\prime} \right|^2$,
			\item[(v) ] $\left| f(t,y,z,v)-f(t,y,z^{\prime},v^{\prime}) \right| \leq \kappa \left( \left\| z-z^{\prime} \right\|+\|   v-v^{\prime}\|_Q \right)$,
			\item[(vi) ] $\left| f(t,y,0,0)   \right| \leq \varphi_t +\kappa  \left|y \right|$ and $\left| g(t,y)   \right| \leq \psi_t+ \kappa \left| y  \right|$,
			\item[(vii) ] $y \mapsto f(t,y,z,v)$ and $y \mapsto g(t,y)$ are continuous, for all $(z,v)\in \mathbb{R}^d \times \mathbb{L}^2_Q$, $\mathbb{P}$-a.s.
		\end{itemize}
		
		\item[\textbf{(H3) }] The reflecting barrier $L:=\left(L_t\right)_{0 \leq t \leq T}$ is an $\mathbb{F}$-progressively measurable, RCLL process  satisfying:
		\begin{itemize}
			\item[$\bullet$] $\mathbb{P}$-a.s.  $L_T \leq \xi$,
			\item[$\bullet$] $\mathbb{E}\left[ \sup_{0 \leq t \leq T} \left| e^{\mu A_t} L_t \right|^2  \right]<\infty $.
		\end{itemize}
	\end{itemize}
\begin{remark}
	\begin{itemize}
		\item[$\bullet$] Let $\left(Y_t\right)_{0 \leq t \leq T}$ be some RCLL process in the space $M^{2,A}_{\mu}(\mathbb{R})$. Using the Cauchy-Schwarz inequality, an integration by parts formula, and the linear growth of the generator $g$ (assumption \textbf{(H2)}-(vi)), we may easily derive for any $\mu >0$ and $t \in [0,T]$,
		$$
		\mathbb{E}\left[ \left( \int_{0}^{t}g(s,Y_s) dA_s\right)^2\right]  \leq \frac{2}{\mu}\mathbb{E}\left[  \int_{0}^{t}e^{\mu A_s} \left\lbrace \left|\psi_s\right|^2+\kappa^2 \left|Y_s\right|^2 \right\rbrace dA_s  \right].
		$$
		Similarly, under assumption \textbf{(H2)}-(vi), we can derive for any $\gamma>0$,
		$$
		\mathbb{E}\left[\left( \int_{0}^{t}f(s,Y_s,0,0) ds \right)^2\right]  \leq\frac{2}{\gamma}\mathbb{E}\left[  \int_{0}^{t}e^{\gamma s} \left\lbrace \left|\varphi_s\right|^2+\kappa^2 \left|Y_s\right|^2 \right\rbrace ds \right]. 
		$$
		
		%In particular, under assumption \textbf{(H2)}-(ii), we derive that $\int_{0}^{T} \psi_t dA_t$ is a square-integrable $\mathcal{F}_T$-measurable random variable.
		
		\item[$\bullet$] Similarly, exploiting the fact that $\psi$ is a process with values in $[1,\infty)$ and Hypothesis \textbf{(H2)}-(ii), we obtain:
		$$
		\mathbb{E}\left[e^{\mu A_T}\right] \leq 1+\mu \mathbb{E}\left[\int_{0}^{T}e^{\mu A_s} \psi^2_s dA_s\right]<\infty.
		$$
	\end{itemize}
	\label{Rmq 1}
\end{remark}
\begin{remark}
	According to Proposition 10.19 in \cite{kallenberg1997foundations}, the quasi-left continuity of the filtration ensures that the additional martingale component arising from Lemma \ref{Lemma for the representation in general} is also quasi-left continuous. In particular, due to the continuity of the dual predictable projection of the integer-valued random measure $N$, it follows that every $\mathbb{F}$-local martingale is quasi-left continuous.
\end{remark}

\begin{remark}
	Throughout this paper, $\mathfrak{c}$ will denote a constant that may vary from one line to another. Additionally, the notation $\mathfrak{c}_{\gamma}$ will be employed to emphasize the dependence of the constant $\mathfrak{c}$ on a specific set of parameters $\gamma$.
\end{remark}

	\section{Generalized reflected BSDEs with one RCLL barrier}
	\label{Generalized BSDEs with one reflecting barrier : Existence and uniqueness result}
	\subsection{Formulation}
	The aim of this section is to provide the existence and uniqueness results of a solution for the reflected generalized BSDEs (GRBSDE, for short) with one lower RCLL barrier associated with $(\xi,f,g,A,L)$. Namely, we consider the GRBSDE:
	\begin{equation}
		\left\{
		\begin{split}
			\text{(i)} &~Y_t= \xi+\int_t^T f(s,Y_s,Z_s,V_s)ds+\int_t^T g(s,Y_s)\mathrm{d}A_s+K_T-K_t-\int_t^T Z_s d B_s\\
			& \qquad  -\int_t^T \int_E V_s(e)\tilde{N}(ds,de)-\int_t^T dM_s, \quad 0 \leq t \leq T.\\
			\text{(ii)} &~ Y_t \geq L_t,\ t\leq T. \\
			\text{(iii)} &~  \text{ If } K^{c} \text{ is the continuous part of } K, \text{ then }
			\int_0^T (Y_t-L_t)dK^{c}_t=0.\\
			\text{(iv)} &~ \text{ If } K^{d} \text{ is the purely discontinuous part of } K,
			\text{ then } K^{d}_t=\sum_{0 < s \leq t}(Y_s-L_{s-})^{-}.
		\end{split}
		\right.
		\label{basic equation for one lower reflecting barrier--General case}
	\end{equation}
	
	\begin{remark}
		\begin{itemize}
			\item[(i)] The role of the predictable increasing processes $K=K^{c}+K^{d}$ is to keep the state process $Y$ above the lower reflecting barrier $L$ with minimal energy, meaning it acts only when necessary to prevent $Y$ from crossing the lower obstacle $L$. Specifically, the process $K=K^{c}+K^{d}$ increases only in two cases:
			
			\begin{enumerate}
				\item When the process $Y$ has a negative predictable jump, i.e., when $Y_{-}=L_{-}$, which occurs at a negative predictable jump point of $L$. In this situation, the role of the jump part $K^{d}$ is to make the necessary jump size to $Y$ to keep it above $L$. Hence, when $Y$ has a negative jump point at a predictable time $\tau$, we must have $\Delta K^{d}_{\tau}=(Y_{\tau}-L_{\tau-})^{-}\mathds{1}_{ \{ \Delta L_{\tau}<0 \} }$, and $ Y_{\tau-}=L_{\tau-}$.
				
				\item When the process $Y$ reaches the lower barrier $L$ and tries to prevent it, either in its continuity or in its right positive jump point. In this case, the process $K^{c}$ does the job and acts to keep $Y$ above $L$, satisfying $\int_0^T (Y_s-L_s)dK^{c}_s=0$.
			\end{enumerate}
			\item[(ii)] The Skorokhod condition \eqref{basic equation for one lower reflecting barrier--General case}-(iv) is equivalent to the following property:
			\begin{equation}
				\int_0^T (Y_{s-}-L_{s-})dK_s=0.
				\label{Skorohod condition}
			\end{equation}
			Indeed, if the points \eqref{basic equation for one lower reflecting barrier--General case}-(iii)-(iv) are satisfied, then
			\begin{equation*}
				\begin{split}
					& \int_0^T (Y_{s-}-L_{s-})dK_s=\int_0^T (Y_{s}-L_{s})dK^{c}_s+\int_0^T (Y_{s-}-L_{s-})dK^{d}_s=0.
				\end{split}
			\end{equation*}
			This is due to the fact that, on one hand, the processes $Y$ and $L$ are RCLL, and the process $K^{c}$ increases only when $Y_s=L_s$. On the other hand, the jumps of $K^{d}$ are predictable, and when they occur, we must have $Y_{s-}=L_{s-}$. Conversely, if $\int_0^T (Y_{s-}-L_{s-})dK_s=0$, then $\int_0^T (Y_{s}-L_{s})dK^{c}_s=0$, and $\int_0^T (Y_{s-}-L_{s-})dK^{d}_s=\sum_{0 < t \leq T} (Y_{s-}-L_{s-})\Delta K^{d}_s=0$. The last equality combined with the equation \eqref{basic equation for one lower reflecting barrier--General case}-(i) implies that $\Delta K^{d}_s=(L_{s-}-Y_s)^{+}\mathds{1}_{\{Y_{s-}=L_{s-}\}}$.
			\label{Propertie of the increasing processes for the lower case}
		\end{itemize}
	\end{remark}
	\subsection{Existence and uniqueness}
	The main result of the current section is illustrated in the following theorem:
	\begin{theorem}
		Under the conditions \textbf{(H1)}, \textbf{(H2)}, and \textbf{(H3)}, the GRBSDE \eqref{basic equation for one lower reflecting barrier--General case} has a unique solution $(Y,Z,V,K,M) \in \mathcal{S}^{2}_{\mu}(A;\mathbb{R}) \times M^2_{\mu}(\mathbb{R}^d) \times M^2_{\mu}(\mathbb{L}^2_Q) \times \mathcal{S}^2 \times \mathbb{M}_{\mu}^2$.
		\label{Theorem : GRBSDE with one lower barrier--Main results}
	\end{theorem}
	\begin{proof}
		The proof of Theorem \ref{Theorem : GRBSDE with one lower barrier--Main results} is obtained in two main steps:
		\begin{enumerate}
			\item First, we study a special case of generalized reflected BSDEs where the coefficient $f$ depends only on $y$, i.e., $\mathfrak{f}(t,y):=f(t,y,z,v)$ for any $(t,y,z,v) \in [0,T] \times \mathbb{R} \times \mathbb{R}^d \times \mathbb{R}$.
			
			\item Secondly, using the result from the first part and a fixed-point argument in an appropriate Banach space, we establish the result in the general case where the driver $f$ depends on $(y,z,v)$.
		\end{enumerate}
		
		\paragraph*{Part 1: Case of a driver $f$ independent of $(z,v)$}
		The main finding of this part is stated as follows:
		\begin{theorem}\label{th f ind zv}
			Assume that $(\xi,\mathfrak{f},g,A)$ and $(L_t)_{0 \leq t \leq T}$ satisfies conditions \textbf{(H1)}, \textbf{(H2)}, and \textbf{(H3)}. The GRBSDE \eqref{basic equation for one lower reflecting barrier--General case} associated with $(\xi,\mathfrak{f},g,A,L)$ admits a unique solution $(Y,Z,V,K,M) \in \mathcal{S}^{2}_{\mu}(A;\mathbb{R}) \times M^2_{\mu}(\mathbb{R}^d) \times M^2_{\mu}(\mathbb{L}^2_Q) \times \mathcal{S}^2 \times \mathbb{M}_{\mu}^2$.
		\end{theorem}
		
		\begin{proof}
			The proof of Theorem \ref{th f ind zv} is performed in seven steps.
			\subparagraph*{Step 1: Construction of penalized generalized BSDE for the case when f does not depend on $(z,v)$}
			
			For each $n \in \mathbb{N}$, we consider the following penalized version of generalized BSDE,
			\begin{equation}
				\begin{split}
					& Y^n_t= \xi+\int_t^T \mathfrak{f}_n(s,Y^n_s)ds+\int_t^T g(s,Y^n_s)\mathrm{d}A_s-\int_t^T Z^n_s d B_s -\int_t^T \int_E V^n_s(e)\tilde{N}(ds,de)-\int_t^T dM^n_s, \quad 0 \leq t \leq T.
					\label{penalized version for the RBSDE with one lower barrier}
				\end{split}
			\end{equation}
			With $\mathfrak{f}_n(s,y)=\mathfrak{f}(s,y)+n(y-L_s)^{-}$ for each $n \in \mathbb{N}$. Using Theorem 4.3 from \cite{ELMandELO}, we establish the existence of a unique solution $(Y^n,Z^n,V^n,M^n) \in \mathcal{S}^2_{\mu}(A;\mathbb{R}) \times M^2_{\mu}(\mathbb{R}^{d}) \times M^2_{\mu}(\mathbb{L}^2_Q) \times \mathbb{M}^2_{\mu}$ of the GBSDE \eqref{penalized version for the RBSDE with one lower barrier} associated with data $(\xi,\mathfrak{f}_n,g,A)$.
			
			Denote $K^n_t := n \int_0^t (Y^n_s -L_s)^{-} ds$ for $t \in [0,T]$.
			\subparagraph*{Step 2: Uniform estimation for the sequence $\left\lbrace Y^n,Z^n,V^n,K^n,M^n\right\rbrace_{n \in \mathbb{N}}$}
			The main objective of this Step is to prove the following Lemma:
			\begin{lemma}
				There exists a constant $C>0$ independent of $n$ such that for any $\mu > 0$
				\begin{equation}\label{uniform estimation for the penalized GBSDE with one lower barrier}
					\begin{split}
						\sup_{n \geq 0} \mathbb{E} &\left[\sup_{0 \leq t \leq  T} e^{\mu A_t}\left| Y^n_t \right|^2
						+  \int_0^T e^{\mu A_s}\left\lbrace \left| Y^n_s \right|^2 dA_s+ \left(\| Z^n_s\|^2+\left\| V^n_s \right\|^2_Q\right)  ds\right\rbrace   + \int_0^T e^{\mu A_s} d[M^n ]_s +(K^n_T)^2\right]\leq \mathfrak{c}.
					\end{split}
				\end{equation}
				\label{Lemma of UI}
			\end{lemma}
			\begin{proof}
				For any $\gamma, \mu > 0$, set $\Phi^{\gamma,\mu}_t = e^{\gamma t + \mu A_t}$ for $t \in [0,T]$.
				
				By It\^o's formula (see \cite{Protter}, Theorem II.32), applied to the semimartingale $\Phi^{\gamma,\mu}_t \left|Y^n_t\right|^2$, we can write
				\begin{equation}
					\begin{split}
						& \Phi^{\gamma,\mu}_t \left| Y^n_t \right|^2+ \int_t^T \Phi^{\gamma,\mu}_s \left| Y^n_s \right|^2  \left(\gamma ds+\mu dA_s\right)+\int_t^T e^{ \mu  A_s}\left( \left\lbrace  \| Z^n_s\|^2+ \left\| V^n_s \right\|^2_Q\right\rbrace   ds+d[ M^n ]_s \right) \\
						&= \Phi^{\gamma,\mu}_T \left| \xi \right|^2
						+2 \int_t^T \Phi^{\gamma,\mu}_s Y^n_s  \mathfrak{f}(s,Y^n_s) ds+ 2 \int_t^T \Phi^{\gamma,\mu}_s Y^n_s  dK^{n}_s +2 \int_t^T \Phi^{\gamma,\mu}_s Y_s^n g(s,Y^n_s) dA_s-2 \int_t^T \Phi^{\gamma,\mu}_s Y_s^n Z^n_s  dB_s\\
						&\qquad  -\int_t^T\Phi^{\gamma,\mu}_s
						\int_{E} \left( \left| Y^n_{s-} + V^n_s(e) \right|^2-\left| Y^n_{s-} \right|^2 \right)\tilde{N}(ds,de)-2\int_t^T \Phi^{\gamma,\mu}_s Y_{s-}^n  dM^n_s.
					\end{split}
					\label{Itos formula for GRBSDEPEN1}
				\end{equation}
				Under assumption \textbf{(H2)} and inequalities $a(a-b)^{-} \leq b^{+}(a-b)^{-}$ and $2ab \leq \epsilon a^2 + \dfrac{1}{\epsilon} b ^2$ for every $\epsilon >0$, we have
				$$
				2Y^n_s \mathfrak{f}(s,Y^n_s) \leq (2 \left| \alpha \right| +1) \left| Y^n_s \right|^2 + \left| \varphi_s \right|^2, \quad 2Y^n_s g(s,Y^n_s) \leq  \epsilon_1  \left| Y^n_s \right|^2+ \frac{1}{\epsilon_1}  \psi^2_s,
				$$
				and
				$$
				2 \mathfrak{c} e^{\mu A_s}Y^n_s(Y^n_s-L_s)^{-} \leq \epsilon_2 \mathfrak{c}^2 e^{2 \mu A_s} (L_s^{+})^2 + \frac{1}{\epsilon_2} \left| (Y^n_s-L_s)^{-} \right|^2,\quad \forall \mathfrak{c}>0.
				$$
				%  	for any $\mathfrak{c}>0$.
				
				Plugging these inequalities into (\ref{Itos formula for GRBSDEPEN1}), choosing $\gamma = 2|\alpha| + 2$, $\epsilon_1 = \dfrac{\mu}{2}$, and taking the expectation on both sides yields the existence of a constant $\mathfrak{c}_{\mu} > 0$ depending on $\mu$ such that
				\begin{equation}
					\begin{split}
						&  \mathbb{E}\left[\Phi^{\gamma,\mu}_t \left| Y^n_t \right|^2 \right]+\mathbb{E}\left[ \int_t^T \Phi^{\gamma,\mu}_s\left| Y^n_s \right|^2 \left(ds+dA_s\right) \right]
						+\mathbb{E}\left[ \int_t^T \Phi^{\gamma,\mu}_s \left(  \left\lbrace \| Z^n_s\|^2+\left\| V^n_s \right\|^2_Q \right\rbrace  ds+d[M^n ]_s\right) \right] \\
						&  \leq  \mathfrak{c}_{\mu}\left(  \mathbb{E}\left[ \Phi^{\gamma,\mu}_T \left| \xi \right|^2 \right]
						+ \mathbb{E}\left[ \int_t^T \Phi^{\gamma,\mu}_s \left|\varphi_s \right|^2 ds \right]+\frac{1}{\epsilon_1} \mathbb{E}\left[ \int_t^T e^{\mu A_s}  \psi^2_s dA_s \right] \right)  \\
						&\qquad
						+ \epsilon_2 \mathfrak{c}^2_{\mu} \mathbb{E}\left[ \sup_{0 \leq t \leq T} \Phi^{2\gamma,2\mu}_t (L_t^{+})^2  \right]+\frac{1}{\epsilon_2} \mathbb{E}\left[ (K^n_T-K^n_t)^2 \right].
					\end{split}
					\label{Place in what we substitute for the classical proof}
				\end{equation}
				From equations \eqref{penalized version for the RBSDE with one lower barrier}, by squaring and considering the linear growth of $\mathfrak{f}$ and $g$, and Remark \ref{Rmq 1}, we conclude that
				\begin{equation}
					\begin{split}
						\mathbb{E}\left[ (K^n_T-K^n_t)^2 \right] 
						\quad &\leq 7 \mathbb{E} \left[  \left| \xi\right|^2 +   \left| Y^n_t \right|^2
						+ \frac{2}{\gamma} \int_0^T  e^{\gamma s}\left| \varphi_s \right|^2 ds +\frac{2\kappa^2}{\gamma} \int_t^T e^{\gamma s} \left| Y^n_s \right|^2  ds
						+\frac{2}{\mu}\int_0^T e^{\mu A_s} \left| \psi_s \right|^2 dA_s \right. \\
						& \qquad\qquad \left.   +\frac{2\kappa^2}{\mu}\int_t^T e^{\mu A_s} \left| Y^n_s \right|^2  dA_s    +\int_t^T \left(  \left\lbrace  \| Z^n_s\|^2+\left\| V^n_s \right\|^2_Q \right\rbrace  ds
						+ d[M^n ]_s\right) \right].
					\end{split}
					\label{substituting in the classical proof}
				\end{equation}
				Substituting equation (\ref{substituting in the classical proof}) to equation (\ref{Place in what we substitute for the classical proof}) and choosing $\epsilon_2>7\max\left(\dfrac{1}{\gamma};\dfrac{2\kappa^2}{\gamma};\dfrac{2}{\mu};\dfrac{2\kappa^2}{\mu};1 \right)$, we derive the existence of a constant $\mathfrak{c}_{\alpha,\mu}>0$ such that
				\begin{equation}
					\begin{split}
						& \mathbb{E}\left[ \int_0^T  \Phi^{\gamma,\mu}_s  \left| Y^n_s \right|^2 \left( ds+dA_s\right)+  \int_0^T e^{\mu A_s}   \| Z^n_s\|^2  ds  + \int_0^T  \Phi^{\gamma,\mu}_s  \int_{E} \left| V^n_s(e) \right|^2 Q(s, de )\eta(s) ds+ \int_0^T  \Phi^{\gamma,\mu}_s  d[M^n ]_s\right]\\
						&\leq   \mathfrak{c}_{\alpha,\mu}\left(\mathbb{E}\left[   \Phi^{\gamma,\mu}_T\left| \xi \right|^2 \right]+ \mathbb{E}\left[\int_0^T  \Phi^{\gamma,\mu}_s \left| \varphi_s \right|^2 ds   \right] +\mathbb{E}\left[ \int_0^T  \Phi^{\gamma,\mu}_s \psi^2_s dA_s \right]  +  \mathbb{E}\left[ \sup_{0 \leq t \leq T}  \Phi^{2\gamma,2\mu}_s (L_t^{+})^2  \right]
						\right).
					\end{split}
					\label{est}
				\end{equation}
				
				Returning to \eqref{Itos formula for GRBSDEPEN1}, and applying the Burkholder-Davis-Gundy inequality (similar to those on Page 210 of \cite{ELMandELO}), along with the estimates \eqref{substituting in the classical proof} and \eqref{est}, we deduce that
				\begin{equation*}
					\begin{split}
						&\sup_{n \in \mathbb{N}} \mathbb{E}\left[   \sup_{0 \leq t \leq T}   \Phi^{\gamma,\mu}_t \left| Y^n_t \right|^2+  \int_0^T  \Phi^{\gamma,\mu}_s \left| Y^n_s \right|^2 dA_s
						+  \int_0^T  \Phi^{\gamma,\mu}_s \left(  \left\lbrace  \| Z^n_s\|^2+\left\| V^n_s \right\|^2_Q \right\rbrace  ds+ d[M^n ]_s\right)
						+ \left( K^n_T \right)^2 \right]\\
						&\leq   \mathfrak{c} \left(\mathbb{E}\left[   \Phi^{\gamma,\mu}_T\left| \xi \right|^2 \right]+ \mathbb{E}\left[\int_0^T \Phi^{\gamma,\mu}_s \left| \varphi_s\right|^2 ds   \right] +\mathbb{E}\left[ \int_0^T  \Phi^{\gamma,\mu}_s  \psi^2_s dA_s \right]  +  \mathbb{E}\left[ \sup_{0 \leq t \leq T}  \Phi^{2\gamma,2\mu}_t(L_t^{+})^2  \right]  \right).
					\end{split}
				\end{equation*}
				Which completes the proof.
			\end{proof}

			Now, we establish the convergence of the sequence $\left\lbrace Y^n,Z^n,V^n,K^n,M^n\right\rbrace_{n \in \mathbb{N}}$.
			\subparagraph*{Step 3: There exists an $\mathbb{F}$-optional processes $Y:=(Y_t)_{0 \leq t \leq T}$ such that $Y^n \nearrow Y$ in $\mathbb{L}^2_{\mu}(\Omega \times [0,T], d \mathbb{P} \times dt; \mathbb{R}) \cap \mathbb{L}^2_{\mu}(\Omega \times [0,T], d \mathbb{P} \times dA_t; \mathbb{R})$ }
			Obviously, for each $n \in \mathbb{N}$, and all $(s,y)\in [0,T] \times \mathbb{R}$, $f_n(s,y) \leq f_{n+1}(s,y)$. Therefore, from the comparison Theorem \ref{Version comparison theorem -- independ of v}, we have $Y^{n}_t \leq Y^{n+1}_t$, $0 \leq t  \leq T$, $\mathbb{P}$-a.s. Therefore, there exists a right lower semi-continuous process $Y:=(Y_t)_{0 \leq t \leq T}$ such that
			\begin{equation*}
				Y^n_t \nearrow Y_t, \quad 0 \leq t \leq T,\quad \text{a.s.}
				\label{increasing convergence}
			\end{equation*}
			
			From the uniform estimation \eqref{uniform estimation for the penalized GBSDE with one lower barrier}, the monotonic limit theorem and the dominated convergence theorem for Lebesgue integration, and Fatou's lemma, we have
			\begin{equation*}
				\begin{split}
					\mathbb{E}\left[ \sup_{0 \leq t \leq T}  e^{\mu A_s}\left| Y_t \right|^2 +\int_0^T e^{\mu A_s} \left| Y_s \right|^2 dA_s \right]
					&\leq \liminf_{n \rightarrow \infty} \mathbb{E}\left[  \sup_{0 \leq t \leq T} e^{\mu A_t} \left| Y_t^n \right|^2 +\int_0^T e^{\mu A_s} \left| Y^n_s \right|^2 dA_s \right] \leq \mathfrak{c},
				\end{split}
			\end{equation*}
			and
			\begin{equation}
				\lim\limits_{n \rightarrow \infty} \mathbb{E}\left[   \int_0^T e^{\mu A_s} \left| Y^n_s-Y_s  \right|^2 \left(ds+dA_s\right) \right]=0.
				\label{Converegence by the dominate convergence theorem for the GRBSDE}
			\end{equation}
			
			%_-_-_-_-_-_-_______________-_-_-_-_-_-_______________-_-_-_-_-_______________-_-_-_-_-__________________________-_-_-_-_-______________________________
			\paragraph*{Step 4: Let's show that} 
			$$\mathbb{E}\left[ \sup_{0 \leq t \leq T} \left| (Y^n_t-L_t)^{-} \right|^2 \right] \underset{n \rightarrow \infty}{\longrightarrow} 0.$$
			Let $(\bar{Y}^n_t,\bar{Z}^n_t,\bar{V}^n_t,\bar{M}^n_t)_{t \leq T}$ be the solution of the following unconstrained GBSDE:
			\begin{equation}
				\begin{split}
					\bar{Y}_t^n=\xi & +\int_t^T \mathfrak{f}(s,Y^n_s)  ds + \int_t^T g(s,Y^n_s) dA  _s +n\int_t^T (L_s - \bar{Y}^n_s)ds\\
					&-\int_t^T \bar{Z}^n_s  dB_s - \int_t^T \int_{E} \bar{V}^n_s(e) \tilde{N}( ds , de )-\int_t^T  d\bar{M}^n_s, \quad 0 \leq t \leq T.
				\end{split}
				\label{state process for the unconstrained BSDE that we will focus on it}
			\end{equation}
			By the comparison result given by Theorem \ref{Version comparison theorem -- independ of v}, we have  $Y^n_t \geq \bar{Y}^n_t$, for all $0 \leq t \leq T$, $n \geq 0$ a.s. Next, we will focus on the state process $(\bar{Y}^n_t)_{0 \leq t \leq T}$ of the unconstrained BSDE \eqref{state process for the unconstrained BSDE that we will focus on it}. Let $\tau$ be an $\mathbb{F}$-stopping time in $\mathcal{T}_0^T$. Then, applying It\^{o}'s formula and taking the conditional expectation with respect to $\mathcal{F}_{\tau}$, we have
			\begin{equation}
				\begin{split}
					\bar{Y}^{n}_{\tau}=& \mathbb{E}\left[ e^{-n(T-\tau)} \xi +\int_{\tau}^T e^{-n(s-\tau)} \mathfrak{f}(s,Y^{n}_s)ds \right.\\
					&\qquad \left.  +\int_{\tau}^T e^{-n(s-\tau)} g(s,Y^{n}_s)\mathrm{d}A_s +n \int_{\tau}^T e^{-n(s-\tau)} L_s ds \mid \mathcal{F}_{\tau} \right]
				\end{split}
				\label{Convergence theorem -- conditional expectation for new state process}
			\end{equation}

		\begin{lemma}\label{convergence lemma}
			The sequence of random variables
			\begin{equation*}
				X_n := e^{-n(T-\tau)} \xi + n \int_{\tau}^T e^{-n(s-\tau)} L_s \, ds 
				\xrightarrow[n \to +\infty]{}
				\xi \mathds{1}_{\{\tau=T \}} + L_{\tau} \mathds{1}_{\{ \tau < T \}}
			\end{equation*}
			converges almost surely and in $\mathbb{L}^2(\Omega,\mathcal{F},\mathbb{P})$. Moreover, for any sub-$\sigma$-algebra $\mathcal{G} \subset \mathcal{F}$, the conditional expectation $\mathbb{E}[X_n \mid \mathcal{G}]$ also converges in $\mathbb{L}^2(\Omega,\mathcal{F},\mathbb{P})$ to $\mathbb{E}[\xi \mathds{1}_{\{\tau=T\}} + L_{\tau} \mathds{1}_{\{\tau<T\}} \mid \mathcal{G}]$.
		\end{lemma}
		
	Lemma \ref{convergence lemma} has been widely used in the literature on RBSDE theory. We refer, for example, to \cite{ElKaroui} (page 723), \cite{HamadeneO} (page 8), and \cite{Marzougue} (page 9).
	
		\begin{proof}[Proof of Lemma \ref{convergence lemma}]
			We prove this result in several steps.
			
			\paragraph{1. Almost sure pointwise convergence.}
			Consider a fixed \(\omega \in \Omega\). Two cases arise:
			
			\emph{\underline{Case 1:}} $\tau(\omega) = T$. Then $e^{-n(T-\tau(\omega))} = e^0 = 1$ and the integral $\int_{\tau}^T \ast = \int_T^T \ast = 0$. Hence $X_n(\omega) = \xi(\omega)$, and the limit is $\xi(\omega)$, which corresponds exactly to $\xi(\omega)\mathds{1}_{\{\omega : \tau(\omega)=T\}}(\omega)$.
			
			\emph{\underline{Case 2:}} \(\tau(\omega) < T\). Then $e^{-n(T-\tau(\omega))} \to 0$ as $n \to +\infty$. For the second term, perform the change of variable \(u = s - \tau(\omega)\):
			\begin{equation*}
				n \int_{\tau(\omega)}^T e^{-n(s-\tau(\omega))} L_s(\omega) \, ds = n \int_0^{T-\tau(\omega)} e^{-nu} L_{\tau(\omega)+u}(\omega) \, du.
			\end{equation*}
			(This is a standard "Laplace average" of $u \mapsto L_{\tau(\omega)+u}$ for $\omega \in \Omega$).\\
			Set $\mathsf{K}_n(u) = n e^{-nu} \mathbf{1}_{[0,\infty)}(u)$. This is an approximate identity kernel: $\int_0^\infty \mathsf{K}_n(u)du = 1$ and its mass concentrates near $0$ as $n \to \infty$. Since $L$ has RCLL paths on $[0,T]$, the map \(u \mapsto L_{\tau(\omega)+u}(\omega)\) is right-continuous at $u=0$, so $L_{\tau(\omega)+u}(\omega) \to L_{\tau(\omega)+}(\omega)=L_{\tau(\omega)}(\omega)$ as $u \to 0^+$. \\
			Define the function
			\begin{equation*}
				\varrho(u) := L_{\tau(\omega)+u}(\omega), \quad u \in [0, T-\tau(\omega)].
			\end{equation*}
			Since the process $L$ has RCLL paths, the function $\varrho$ is right-continuous on $[0, T-\tau(\omega))$ and has a left limit at $T-\tau(\omega)$. In particular, $\varrho$ is right-continuous at $u=0$, which means
			\begin{equation*}
				\lim_{u \downarrow 0} \varrho(u) = \varrho(0) = L_{\tau(\omega)}(\omega).
			\end{equation*}
			Let us now extend $\varrho$ to $[0, \infty)$ by setting $\varrho(u) = \varrho(T-\tau(\omega))$ for all $u > T-\tau(\omega)$. This extension is bounded on \([0,\infty)\) because \(L\) has RCLL paths on \([0,T]\), hence is bounded on \([\tau(\omega), T]\). Let
			\begin{equation*}
				M_\omega := \sup_{t \in [0,T]} |L_t(\omega)| < \infty.
			\end{equation*}
			Now consider the integral of interest:
			\begin{equation*}
				I_n := n \int_0^{T-\tau(\omega)} e^{-nu} \varrho(u) \, du.
			\end{equation*}
			We can write it as
			\begin{equation*}
				I_n = n \int_0^{\infty} e^{-nu} \varrho(u) \, du - n \int_{T-\tau(\omega)}^{\infty} e^{-nu} \varrho(u) \, du =: J_n - R_n.
			\end{equation*}
			
			For $u \geq T-\tau(\omega) > 0$, we have $|\varrho(u)| \leq M_\omega$. Therefore
			\begin{equation*}
				|R_n| \leq n \int_{T-\tau(\omega)}^{\infty} e^{-nu} M_\omega \, du = M_\omega \left[ -e^{-nu} \right]_{u=T-\tau(\omega)}^{\infty} = M_\omega e^{-n(T-\tau(\omega))}.
			\end{equation*}
			Since $T-\tau(\omega) > 0$, we have $e^{-n(T-\tau(\omega))} \to 0$ as $n \to +\infty$. Thus,
			\begin{equation*}
				\lim_{n \to +\infty} R_n = 0.
			\end{equation*}
			Consequently, $\lim_{n \to +\infty} I_n = \lim_{n \to +\infty} J_n$, provided the latter limit exists.
			
			By definition, we have $J_n = n \int_0^{\infty} e^{-nu} \varrho(u) \, du$. Let $\varepsilon > 0$. By the right-continuity of $\varrho$ at $0$, there exists $\delta > 0$ such that for all $u \in [0, \delta]$,
			\begin{equation*}
				|\varrho(u) - \varrho(0)| < \frac{\varepsilon}{2}.
			\end{equation*}
			We then decompose \(J_n\) into two parts:
			\begin{equation*}
				J_n = n \int_0^{\delta} e^{-nu} \varrho(u) \, du + n \int_{\delta}^{\infty} e^{-nu} \varrho(u) \, du =: A_n + B_n.
			\end{equation*}
			For $A_n$, we can write
			\begin{equation*}
				A_n = \varrho(0) \cdot n \int_0^{\delta} e^{-nu} \, du + n \int_0^{\delta} e^{-nu} (\varrho(u)-\varrho(0)) \, du.
			\end{equation*}
			Computing the first integral,
			\begin{equation*}
				n \int_0^{\delta} e^{-nu} \, du = \left[ -e^{-nu} \right]_{u=0}^{\delta} = 1 - e^{-n\delta} \xrightarrow[n \to +\infty]{} 1.
			\end{equation*}
			For the second integral, use the right-continuity property. For all $u \in [0,\delta]$, we have $|\varrho(u)-\varrho(0)| < \varepsilon/2$. Hence
			\begin{equation*}
				\left| n \int_0^{\delta} e^{-nu} (\varrho(u)-\varrho(0)) \, du \right|
				\leq n \int_0^{\delta} e^{-nu} \cdot \frac{\varepsilon}{2} \, du
				= \frac{\varepsilon}{2} (1 - e^{-n\delta}) \leq \frac{\varepsilon}{2}.
			\end{equation*}
			Thus, for sufficiently large \(n\), we have
			\begin{equation*}
				|A_n - \varrho(0)| \leq |\varrho(0)| \cdot |1 - (1 - e^{-n\delta})| + \frac{\varepsilon}{2}
				= |\varrho(0)| e^{-n\delta} + \frac{\varepsilon}{2}.
			\end{equation*}
			Since $e^{-n\delta} \to 0$, there exists $N_1 \in \mathbb{N}^\ast$ such that for all $n \geq N_1$, $|\varrho(0)| e^{-n\delta} < \varepsilon/2$. Therefore, for $n \geq N_1$,
			\begin{equation*}
				|A_n - \varrho(0)| < \varepsilon.
			\end{equation*}
			This proves that $\lim_{n \to +\infty} A_n = \varrho(0) = L_{\tau(\omega)}(\omega)$.
			
			Now let us control the term $B_n$. For this term, we have
			\begin{equation*}
				|B_n| \leq n \int_{\delta}^{\infty} e^{-nu} |\varrho(u)| \, du
				\leq M_\omega \cdot n \int_{\delta}^{\infty} e^{-nu} \, du
				= M_\omega e^{-n\delta}.
			\end{equation*}
			Since $\delta > 0$, $e^{-n\delta} \to 0$ as $n \to +\infty$. Hence $\lim_{n \to +\infty} B_n = 0$.
			
			Finally, we have $J_n = A_n + B_n$ with $A_n \to \varrho(0)$ and $B_n \to 0$ as $n \to +\infty$. Therefore $J_n \to \varrho(0)$. Since $I_n = J_n - R_n$ and $R_n \to 0$, we obtain
			\begin{equation*}
				\lim_{n \to \infty} I_n = \varrho(0) = L_{\tau(\omega)}(\omega).
			\end{equation*}
			
			Returning to $X_n(\omega)$, recall that
			\begin{equation*}
				X_n(\omega) = e^{-n(T-\tau(\omega))} \xi(\omega)
				+ n \int_{\tau(\omega)}^T e^{-n(s-\tau(\omega))} L_s(\omega) \, ds.
			\end{equation*}
			
			The first term tends to $0$ because $T-\tau(\omega) > 0$. The second term is exactly $I_n$ (after the change of variable $u = s-\tau(\omega)$). We have just shown that $I_n \to L_{\tau(\omega)}(\omega)$. Hence
			\begin{equation*}
				X_n(\omega) \to 0 + L_{\tau(\omega)}(\omega) = L_{\tau(\omega)}(\omega).
			\end{equation*}
			Since we are in the case $\tau(\omega) < T$, we have
			$L_{\tau(\omega)}(\omega) = L_{\tau(\omega)}(\omega) \mathds{1}_{\{\tau < T\}}(\omega)$.
			In the case $\tau(\omega) = T$, we directly have $X_n(\omega) = \xi(\omega)$, which is the announced limit.
			
			Thus, for every $\omega \in \Omega$, we have established the pointwise convergence
			\begin{equation*}
				X_n(\omega) \longrightarrow
				\xi(\omega) \mathds{1}_{\{\tau(\omega)=T\}} +
				L_{\tau(\omega)}(\omega) \mathds{1}_{\{\tau(\omega)<T\}}.
			\end{equation*}
			
			\paragraph{2. Convergence in $\mathbb{L}^2$.}
			First, assume that the process $L$ is uniformly bounded: there exists a constant $C > 0$ such that $\sup_{t \in [0,T]} |L_t| \leq C$ a.s., and $|\xi| \leq C$ a.s. Then
			\begin{equation*}
				|X_n| \leq C + n \int_{\tau}^T e^{-n(s-\tau)} \cdot C \, ds
				\leq C + C(1 - e^{-n(T-\tau)}) \leq 2C.
			\end{equation*}
			The sequence \((X_n)\) is thus uniformly bounded. The almost sure convergence established in \textbf{Step 1}, together with the dominated convergence theorem (in $\mathbb{L}^2$), yields \(X_n \to X\) in $\mathbb{L}^2$, where $X = \xi \mathds{1}_{\{\tau=T\}} + L_{\tau} \mathds{1}_{\{\tau<T\}}$.
			
			In the general case where \(L\) is not necessarily bounded but \(\sup_{t \in [0,T]} |L_t| \in \mathbb{L}^2\) (from assumption \textbf{(H3)}), we use a truncation argument. For any \(M > 0\), define $L_t^M = (L_t \wedge M) \vee (-M)$ for $t \in [0,T]$ and $\xi^M = (\xi \wedge M) \vee (-M)$. Let $X_n^M$ be the corresponding quantity and $X^M$ its limit. Then
			\begin{align*}
				\|X_n - X\|_{\mathbb{L}^2}
				&\leq \|X_n - X_n^M\|_{\mathbb{L}^2}
				+ \|X_n^M - X^M\|_{\mathbb{L}^2}
				+ \|X^M - X\|_{L^2}.
			\end{align*}
			The first term is bounded by
			$2\|\sup_t |L_t| \mathbf{1}_{\{\sup_t |L_t| > M\}}\|_{L^2}
			+ 2\|\xi \mathbf{1}_{\{|\xi| > M\}}\|_{L^2}$,
			which tends to $0$ as $M \to +\infty$ by dominated convergence. The second term tends to $0$ for each fixed $M$ by the bounded case. The third term is bounded similarly to the first. A subsequence extraction allows us to conclude that $X_n \to X$ in $\mathbb{L}^2$.
			
			\paragraph{3. Convergence of conditional expectations.}
			Let $\mathcal{G} \subset \mathcal{F}$ be a sub-$\sigma$-algebra. From \textbf{Step 2}, $X_n \to X$ in $\mathbb{L}^2$, so $\mathbb{E}[|X_n - X|^2] \to 0$. By the conditional Jensen inequality,
			\begin{equation*}
				\mathbb{E}\bigl[|\mathbb{E}[X_n \mid \mathcal{G}] - \mathbb{E}[X \mid \mathcal{G}]|^2\bigr]
				\leq \mathbb{E}\bigl[\mathbb{E}[|X_n - X|^2 \mid \mathcal{G}]\bigr]
				= \mathbb{E}[|X_n - X|^2] \to 0.
			\end{equation*}
			Thus, $\mathbb{E}[X_n \mid \mathcal{G}] \to \mathbb{E}[X \mid \mathcal{G}]$ in $\mathbb{L}^2$. Note that this convergence also holds almost surely along a subsequence by the conditional dominated convergence theorem, but the $\mathbb{L}^2$ convergence is established directly without extraction.
			
			This completes the proof of Lemma \ref{convergence lemma}.
		\end{proof}
		
		By Lemma \ref{convergence lemma}, we have the following convergence:
			\begin{equation*}
				e^{-n(T-\tau)} \xi + n \int_{\tau}^T e^{-n(s-\tau)} L_s \, ds \xrightarrow[n \to +\infty]{}
				\xi \mathds{1}_{\{\tau=T \}}+ L_{\tau} \mathds{1}_{\{ \tau < T \}}
			\end{equation*}
			almost surely and in $\mathbb{L}^2(\Omega,\mathcal{F},\mathbb{P})$, and the conditional expectation also converges in $\mathbb{L}^2(\Omega,\mathcal{F},\mathbb{P})$. Consequently,
			\begin{equation*}
				e^{-n(T-\tau)} L_{\tau} + n \int_{\tau}^T e^{-n(s-\tau)} L_s ds-L_{\tau} \xrightarrow[n \to +\infty]{} 0
			\end{equation*}
			almost surely and conditionally with respect to $\mathcal{F}_{\tau}$ in  $L^2(\Omega,\mathcal{F},d \mathbb{P})$.\\
			Henceforth, define the sequence of processes
			$\{X^n\}_{n \in \mathbb{N}} := \left\{ \left( e^{-n(T-t)}L_T + n \int_t^T e^{-n(s-t)}L_s \, ds - L_t \right)_{t \in [0,T]} \right\}_{n \in \mathbb{N}}$.
			Then $\{X^n\}_{n \in \mathbb{N}}$ converges uniformly in $t$ to zero. The same holds for the sequence $\left\{(X^n)^{-}\right\}_{n \in \mathbb{N}}$. The convergence of $\{X^n\}_{n \in \mathbb{N}}$ and $\left\{(X^n)^{-}\right\}_{n \in \mathbb{N}}$ has been explicitly proved in \cite{StochDyn} (see pages 2450001-21--22). However, for completeness we recall the arguments used.
			
			For $\mathbb{P}$-almost all $\omega$, the function $t \mapsto L_t(\omega)$ is right-continuous. Then, for all $t \in [0, T]$ and for all $\epsilon > 0$, there exists some $\delta(\omega) > 0$ such that for any $s \in ]t, t+\delta(\omega)[$, we have $\left|L_s(\omega) - L_t(\omega)\right| < \epsilon$. Now, for all $t \in [0, T]$, if $T-t < \delta(\omega)$, then
			$$
			\left|X_t^n\right| \leq \left|L_T - L_t\right| \mathds{1}_{\{t<T\}} + \left|L_T - L_t\right| < 2\epsilon
			$$
			for $n > \eta(\delta(\omega))$, when $n$ is sufficiently large.\\
			On the other hand, if $T-t > \delta(\omega)$, we may choose an instant $t_0 \in [t, T]$ such that $t_0 = t + \frac{\delta(\omega)}{2}$. It then follows that
			$$
			\left|X_t^n\right| \leq \left|e^{-n(T-t)} L_T + n \int_{t_0}^T e^{-n(s-t)} L_s \, ds\right|
			+ \left|n \int_t^{t_0} e^{-n(s-t)} L_s \, ds - L_t\right|.
			$$
			From the fact that $S(\omega) := \sup_{0 \leq s \leq T} \left|L_s(\omega)\right|^2 < \infty$ (from assumption \textbf{(H3)}), we deduce the existence of an integer $\eta(\delta(\omega), S(\omega))$ such that
			$$
			\left|X_t^n\right| < 2\epsilon, \quad \forall n > \eta(\delta(\omega), S(\omega)),
			$$
			when $n$ is sufficiently large. Therefore, the sequence $\left(X^n\right)_{n \in \mathbb{N}}$ converges uniformly in $t$ to $0$, $\mathbb{P}$-a.s., and the same also holds for $\left(\left(X^n\right)^{-}\right)_{n \in \mathbb{N}}$.

			Also, we have
			\begin{equation*}
				\mathbb{E}\left[ \left| \int_{\tau}^T e^{-n(s-\tau)} \mathfrak{f}(s,Y^n_s) ds  \right|^2 \right]
				\leq \dfrac{1}{n} \left( \mathbb{E}\left[ \int_0^T e^{\mu A_s} \varphi_s^2 ds \right]
				+ T \kappa^2 \mathbb{E}\left[ \sup_{0 \leq t \leq T} e^{\mu A_s} \left| Y^n_s \right|^2     \right] \right)
			\end{equation*}
			and
			\begin{equation*}
				\begin{split}
					\mathbb{E}\left[ \left| \int_{\tau}^T e^{-n(s-\tau)} g(s,Y^n_s) dA_s  \right|^2 \right] 
					& \quad \leq \mathbb{E}\left[  \left(\int_{\tau}^T   e^{-\mu(A_s-A_{\tau})} dA_s \right)
					\left( \int_{\tau}^T e^{-2n(s-\tau)}e^{\mu(A_s-A_{\tau})}\left| g(s,Y^n_s) \right|^2 dA_s \right) \right] \\
					& \quad  \leq \dfrac{2}{\mu} \left( \mathbb{E}\left[ \int_{0}^T e^{-2n(s-\tau)} e^{\mu A_s} \psi^2_s dA_s
					+\kappa^2 \int_{0}^T e^{-2n(s-\tau)}  e^{\mu A_s}  \left| Y^n_s \right|^2 dA_s   \right] \right)
				\end{split}
			\end{equation*}
			Hence, this implies that
			\begin{equation*}
				\mathbb{E}\left[ \int_{\tau}^T e^{-n(s-\tau)} \mathfrak{f}(s,Y^n_s) ds
				+\int_{\tau}^T e^{-n(s-\tau)} g(s,Y^n_s) dA_s \mid \mathcal{F}_{\tau}  \right] \rightarrow 0
			\end{equation*}
			in $\mathbb{L}^2(\Omega,\mathcal{F},d \mathbb{P})$ as $n \rightarrow +\infty$. Now, we denote
			\begin{equation*}
				\bar{y}^n_t := e^{-n(T-t)}\xi + n \int_t^T e^{-n(s-t)} L_s ds+\int_{t}^T e^{-n(s-t)} \mathfrak{f}(s,Y^n_s) ds+ \int_{t}^T e^{-n(s-t)} g(s,Y^n_s) dA_s.
			\end{equation*}
			Obviously, from \eqref{Convergence theorem -- conditional expectation for new state process}, we have
			$\bar{Y}^n_t-L_t = \mathbb{E}\left[  \bar{y}^n_t-L_t  \mid \mathcal{F}_t \right]$. Then, using Jensen's inequality and Doob's maximal quadratic inequality, we obtain
			\begin{equation}
				\begin{split}
					\mathbb{E}\left[ \sup_{0 \leq t \leq T} \left| (\bar{Y}^n_t-L_t)^{-} \right|^2  \right] &
					\leq \mathbb{E}\left[ \sup_{0 \leq t \leq T} \left| \mathbb{E}\left[
					\left( \bar{y}^n_t-L_t \right)^{-} \mid \mathcal{F}_t \right] \right|^2  \right] \leq 4 \mathbb{E} \left[ \sup_{0 \leq t \leq T}    \left| \left( \bar{y}^n_t-L_t \right)^{-}  \right|^2  \right].
				\end{split}
				\label{Resultst for one RBSDE with one reflecting barrier}
			\end{equation}
			Using \eqref{state process for the unconstrained BSDE that we will focus on it}, we found that
			\begin{equation}
				\begin{split}
					\left| \left( \bar{y}^n_t-L_t \right)^{-} \right|^2 &\leq \left| \left( X^n_t+\int_{t}^T e^{-n(s-t)} \mathfrak{f}(s,Y^n_s) ds
					+ \int_{t}^T e^{-n(s-t)} g(s,Y^n_s) dA_s \right)^{-} \right|^2\\
					&\leq 2  \left| \left( X^n_t\right)^{-} \right|^2 +2  \left| \int_{t}^T e^{-n(s-t)} \mathfrak{f}(s,Y^n_s) ds
					+ \int_{t}^T e^{-n(s-t)} g(s,Y^n_s) dA_s \right|^2.
				\end{split}
				\label{finally lower barrier}
			\end{equation}
			The dominates convergence theorem  implies that
			\begin{equation*}
				\mathbb{E}\left[ \sup_{0 \leq t \leq T}  \left| \left( X^n_t\right)^{-} \right|^2 \right]  \underset{n \rightarrow \infty}{\longrightarrow} 0
			\end{equation*}
			and
			\begin{equation*}
				\mathbb{E}\left[ \sup_{0 \leq t \leq T}  \left| \int_{t}^T e^{-n(s-t)} \mathfrak{f}(s,Y^n_s) ds
				+ \int_{t}^T e^{-n(s-t)} g(s,Y^n_s) dA_s \right|^2 \right]  \underset{n \rightarrow \infty}{\longrightarrow} 0.
			\end{equation*}
			Hence, from inequality \eqref{finally lower barrier} and \eqref{Resultst for one RBSDE with one reflecting barrier}, we deduce that
			\begin{equation*}
				\mathbb{E}\left[ \sup_{0 \leq t \leq T} \left| (\bar{Y}^n_t-L_t)^{-} \right|^2 \right] \underset{n \rightarrow \infty}{\longrightarrow} 0.
			\end{equation*}
			Using the fact that $(Y^n_t-L_t)^{-} \leq (\bar{Y}^n_t-L_t)^{-} $; for any $t \leq T$, $\mathbb{P}$-a.s., we conclude that
			\begin{equation*}
				\mathbb{E}\left[ \sup_{0 \leq t \leq T}  \left| (Y^n_t-L_t)^{-} \right|^2 \right] \underset{n \rightarrow \infty}{\longrightarrow} 0.
			\end{equation*}
			%We get the desired result.
			\subparagraph*{Step 5: Convergence of the  sequence $\left\lbrace Y^n,Z^n,V^n,K^n,M^n\right\rbrace_{n \in \mathbb{N}}$}
			The objective of this section is to establish the following convergence result:
			\begin{lemma}
				\begin{itemize}
					\item The limited process $Y:=(Y_t)_{0 \leq t \leq T}$ defined in \eqref{increasing convergence} satisfies 
					\begin{equation*}
						\begin{split}
							\lim\limits_{n \rightarrow \infty} \mathbb{E} \left[ \sup_{0 \leq t \leq T} \left| Y^n_t-Y_t \right|^2  \right]=0,
						\end{split}
					\end{equation*}
					in particular $Y \in \mathcal{D}^2$.
					\item There existence a quadruplet of  $\mathbb{F}$-adapted processes $(Z,V,K,M):=(Z_t,V_t,K_t,M_t)_{0 \leq t \leq T}$ such that 
					\begin{equation*}
						\begin{split}
							& \lim\limits_{n \rightarrow \infty} \mathbb{E} \left[ \int_0^T  \|Z^n_s-Z_s \|^2 ds
							+ \int_0^T   \left\| V^n_s(e)-V(e)_s \right\|^2_Qds  +\int_0^T d[M^n-M]_s+\sup_{0 \leq t \leq T} \left| K^n_t-K_t \right|^2   \right]=0,
						\end{split}
					\end{equation*}
					and $(Y,Z,V,K,M)$ satisfies the following GBSDE:
					\begin{equation*}
						\begin{split}
							&Y_t=\xi+\int_t^T \mathfrak{f}(s,Y_s)ds+\int_t^T g(s,Y_s)dA_s+(K_T+K_t) -\int_t^T Z_s dB_s -\int_t^T \int_E V_s(e) \tilde{N}(ds,de)-\int_t^T dM_s.
						\end{split}
					\end{equation*}
				\end{itemize}
				\label{Lemma 0316}
			\end{lemma}
			\begin{proof}
				For any $n > p \geq 0$, applying It\^{o}'s formula, assumption \textbf{(H2)}, and taking expectations, we have
				\begin{equation*}
					\begin{split}
						& 
						(\gamma-2|\alpha|) \mathbb{E}\int_t^T e^{\gamma s} \left| Y^n_s-Y^p_s  \right|^2 ds+\mathbb{E}\int_t^T e^{\gamma s}\left( \left\lbrace  \| Z^n_s-Z^p_s  \|^2+\left\| V^n_s-V^p_s  \right\|^2_Q\right\rbrace  ds+d[M^n-M^p]_s\right) \\
						&\leq  2e^{\gamma T}\mathbb{E}\int_0^T  \left( Y^n_s-L_s  \right)^{-} dK^p_s
						+2e^{\gamma T}\mathbb{E}\int_0^T \left( Y^p_s-L_s  \right)^{-} dK^n_s.
					\end{split}
				\end{equation*}
				Using the results of \textbf{Step 2} and \textbf{Step 4} and the Cauchy-Schwarz inequality, we deduce that
				\begin{equation}
					\begin{split}
						\mathbb{E}\int_t^T (Y^p_s-L_s)^{-} dK^n_s
						\leq \left(\mathbb{E}\left[  \sup_{0 \leq t \leq T} \left((Y^p_s-L_s)^{-} \right)^2 \right]\right)^\frac{1}{2}
						\left(\mathbb{E}\left|K^n_T \right|^2 \right)^\frac{1}{2} \xrightarrow[n,p \rightarrow +\infty]{} 0.
					\end{split}
					\label{first convergence results for Z,V,M,K in GRBSDE1}
				\end{equation}
				By the same arguments
				\begin{equation}
					\begin{split}
						\mathbb{E}\int_t^T (Y^n_s-L_s)^{-} dK^p_s  \xrightarrow[n,p \rightarrow +\infty]{} 0,\quad 0 \leq t \leq T.
					\end{split}
					\label{second convergence results for Z,V,M,K in GRBSDE1}
				\end{equation}
				Putting \eqref{first convergence results for Z,V,M,K in GRBSDE1} and \eqref{second convergence results for Z,V,M,K in GRBSDE1} together with the result of \textbf{Step 2} and choosing $\gamma\geq 2|\alpha|$, we conclude that
				\begin{equation*}
					\| Y^n-Y^p \|_{M^{2}(\mathbb{R})}^2 + \| Z^n-Z^p \|_{M^2(\mathbb{R}^d)}^2+\| V^n-V^p \|_{M^2(\mathbb{L}^2_Q)}^2+\| M^n-M^p \|_{\mathbb{M}^2}^2 \xrightarrow[n,p \rightarrow \infty]{} 0.
				\end{equation*}
				Henceforth, $\left\lbrace Y^n,Z^n,V^n,M^n\right\rbrace_{n \in \mathbb{N}}$ is a Cauchy sequence in $M^{2}(\mathbb{R}) \times M^2(\mathbb{R}^d)\times M^2(\mathbb{L}^2_Q) \times \mathbb{M}^2$. Consequently, there exists a triplet of processes $(Z,V,M) \in M^2(\mathbb{R}^d)\times M^2(\mathbb{L}^2_Q) \times \mathbb{M}^2$ such that:
				\begin{equation*}
					\| Z^n-Z \|_{M^2(\mathbb{R}^d)}^2+\| V^n-V \|_{M^2(\mathbb{L}^2_Q)}^2+\| M^n-M \|_{\mathbb{M}^2}^2 \xrightarrow[n \rightarrow +\infty]{} 0.
				\end{equation*}
				Then by the B-D-G inequality, it follows that
				\begin{equation}
					\begin{split}
						&2\mathbb{E}\left[ \sup_{0 \leq t \leq T} \left| \int_t^T (Y^n_s-Y^p_s)( Z^n_s-Z^p_s) dB_s  \right|  \right]\leq \frac{1}{6} \left\| Y^n-Y^p \right\|^2_{\mathcal{D}^2}
						+6c^2 \mathbb{E}\left[ \int_0^T \| Z^n_s-Z^p_s  \|^2 ds \right].
					\end{split}
					\label{BDG GBSDE(f,L)Z}
				\end{equation}
				By the same arguments we have
				\begin{equation}
					\begin{split}
						&
						2\mathbb{E}\left[ \sup_{0 \leq t \leq T} \left| \int_t^T \int_{E}   (Y^n_{s-}-	Y^p_{s-})( V^n_s(e)-V_s^p(e)) \tilde{N}(ds,de)  \right|  \right] \leq \frac{1}{6}  \left\| Y^n-Y^p \right\|^2_{\mathcal{D}^2}
						+6c^2 \mathbb{E}\left[ \int_0^T   \left\| V^n_s-V^p_s\right\|^2_Q ds \right],
					\end{split}
					\label{BDG GBSDE(f,L)V}
				\end{equation}
				and
				\begin{equation}
					\begin{split}
						&\mathbb{E}\left[ \sup_{0 \leq t \leq T} \left| \int_t^T  (Y^n_{s-}-Y^p_{s-})  d(M^n-M^p)_s  \right|  \right]
						\leq \frac{1}{6} \left\| Y^n-Y^p \right\|^2_{\mathcal{D}^2}
						+6c^2 \mathbb{E}\left[\int_0^T  d[ M^n-M^p ]_s \right].
					\end{split}
					\label{BDG GBSDE(f,L)M}
				\end{equation}
				Next, plugging the estimates \eqref{BDG GBSDE(f,L)Z}, \eqref{BDG GBSDE(f,L)V}, and \eqref{BDG GBSDE(f,L)M} into It\^{o}'s formula with the semimartingale $ \left|Y^n_s-Y^p_s \right|^2$ on $[t,T]$ implies the existence of a universal constant $\mathfrak{c}>0$ such that
				\begin{equation*}
					\begin{split}
						\left\| Y^n-Y^p \right\|^2_{\mathcal{D}^2}&=\mathbb{E}\left[ \sup_{0 \leq t \leq T}  \left| Y^n_t-Y^p_t \right|^2 \right]\\
						& \leq \mathfrak{c} \left( \mathbb{E}\left[ \int_0^T  \left| Y^n_s-Y^p_s \right|^2 ds \right]+ \| Z^n-Z^p \|_{M^2(\mathbb{R}^d)}^2+\| V^n-V^p \|_{M^2(\mathbb{L}^2_Q)}^2+\| M^n-M^p \|_{\mathbb{M}^2}^2\right.\\
						& \qquad \qquad  \left.  +\mathbb{E}\left[ \int_0^T \left( Y^n_s-L_s  \right)^{-} dK^p_s  \right]  +\mathbb{E}\left[ \int_0^T  \left( Y^p_s-L_s  \right)^{-} dK^n_s  \right] \right) \xrightarrow[n,p \rightarrow +\infty]{} 0.
					\end{split}
				\end{equation*}
				Then
				\begin{equation*}
					\lim \limits_{n \rightarrow +\infty} \mathbb{E}\left[ \sup_{0 \leq t \leq T} \left| Y^n_t-Y_t \right|^2 \right] = 0, \quad \text{ and } Y \in \mathcal{D}^2.
					\label{Convergence in Yn with JSM }
				\end{equation*}
				On the other hand, we have,  $Y^0_t \leq Y^n_t \leq Y_t$, for all $0 \leq t \leq T$ and all $n \in \mathbb{N}$. Then,
				$\left| Y^n_t \right| \leq  \left| Y^0_t \right| \vee \left| Y_t \right| $, $\forall n \in \mathbb{N}$ with $Y^0, Y \in \mathcal{D}^2 \cap M^{2,A}$. Using the linear growth of $\mathfrak{f}$ and $g$, is follows that
				\begin{itemize}
					\item $\sup_{n \geq 0} \left| \mathfrak{f}(s,Y^n_s) \right| \leq \bar{\varphi}_s + \kappa \left\{ \left| Y^0_t \right| \vee \left| Y_t \right| \right\}$ with $\mathbb{E}\left[ \int_0^T \bar{\varphi}_s^2 ds   \right]< \infty$,
					\item $\sup_{n \geq 0} \left| g(s,Y^n_s) \right| \leq \psi_s + \kappa \left\{ \left| Y^0_t \right| \vee \left| Y_t \right| \right\}$ with $\mathbb{E}\left[ \int_0^T \psi_s^2 dA_s   \right]< \infty$.
				\end{itemize}
				Furthermore, by \eqref{Converegence by the dominate convergence theorem for the GRBSDE} and the continuity of the random functions $f$, $g$ in $y$, we get
				$\mathfrak{f}(s,Y^n_s)  \xrightarrow[n \rightarrow \infty]{}  \mathfrak{f}(s,Y_s)$ and $
				g(s,Y^n_s)  \xrightarrow[n \rightarrow \infty]{}  g(s,Y_s)$ for almost all $(t,\omega) \in \Omega \times [0,T]$. Hence, from the dominate convergence theorem, we deduce that,  $\mathbb{P}$-a.s., $\forall t \in [0,T]$,
				\begin{itemize}
					\item $\mathfrak{f}(s,Y^n_s)ds \xrightarrow[n \rightarrow \infty]{}  \mathfrak{f}(s,Y_s)ds$ in $\mathbb{L}^2(\Omega \times [0,T],d\mathbb{P}\otimes dt;\mathbb{R})$,
					\item $ g(s,Y^n_s)dA_s \xrightarrow[n \rightarrow \infty]{}  g(s,Y_s)dA_s$ in $\mathbb{L}^2(\Omega \times [0,T],d\mathbb{P}\otimes dA_t;\mathbb{R})$.
				\end{itemize}
				Next, a standard calculations by the means of the B-D-G inequality, Cauchy-Schwarz inequality, implies that, for some constant $\mathfrak{c}>0$,
				\begin{equation*}
					\begin{split}
						&\mathbb{E}\left[ \sup_{0 \leq t \leq T} \left| K^n_t-K^p_t \right|^2 \right]\\
						& \leq \mathfrak{c} \left(\mathbb{E}\left[ \sup_{0 \leq t \leq T} \left| Y^n_t-Y^p_t \right|^2 \right]+  \mathbb{E}\left[ \int_0^T \left|  \mathfrak{f}(s,Y^n_s)-\mathfrak{f}(s,Y^p_s) \right|^2 ds \right]   +\mathbb{E}\left[ \int_0^T e^{\mu A_s}\left|g(s,Y^n_s)-g(s,Y^p_s) \right|^2 dA_s \right] \right.\\
						&\qquad \qquad \left.+\|Z^n-Z^p \|^2_{M^2(\mathbb{R}^d)}+\|V^n-V^p \|^2_{M^2(\mathbb{L}^2_Q)}+\|M^n-M^p \|^2_{\mathbb{M}^2} \right) \xrightarrow[n,p \rightarrow +\infty]{} 0
					\end{split}
				\end{equation*}
				Consequently, there exists an $\mathbb{F}$-predictable non-decreasing  process $K:=(K_t)_{0 \leq t \leq T}$ ($K_0=0$) such that
				\begin{equation*}
					\lim\limits_{n \rightarrow \infty} \mathbb{E}\left[ \sup_{0 \leq t \leq T} \left| K^n_t-K_t \right|^2 \right]=0
				\end{equation*}
				Finally, it is easy to see that the limiting process $(Y_t,Z_t,V_t,K_t,M_t)_{0 \leq t \leq T}$ satisfies
				\begin{equation}
					\begin{split}
						&Y_t=\xi+\int_t^T \mathfrak{f}(s,Y_s)ds+\int_t^T g(s,Y_s)dA_s+(K_T-K_t) -\int_t^T Z_s dB_s -\int_t^T \int_E V_s(e) \tilde{N}(ds,de)-\int_t^T dM_s.
					\end{split}
					\label{equation satified by the limitting process Y : one reflected case}
				\end{equation}
				Which completes the proof of the Lemma \ref{Lemma 0316}.
			\end{proof}

			Now, we demonstrate the minimality condition ensured by the limiting process $\left(K_t\right)_{0 \leq t \leq T}$.
			\subparagraph*{Step 6: Verification of the Skorokhod Condition}
			\begin{lemma}
				The non-decreasing process $K:=(K_t)_{0 \leq t \leq T}$ is RCLL and satisfies,
				\begin{equation*}
					\int_0^T (Y_s-L_s)dK^{c}_s=0, \quad\text{ and }\quad K^d_t=\sum_{0 < s \leq t}(Y_s-L_{s-})^{-}\mathds{1}_{\{\Delta L_{s}<0\}} \text{ for any } t \in [0,T],
				\end{equation*}
				where $K^c$ is the continuous (non-decreasing) non-part of $K$ and and $K^d$ it's the  purely discontinuous part (predictable and non-decreasing) such that $K=K^c+K^d$.
				\label{Lemma : Skorokhod condition of the GBSDE with one lower barrier}
			\end{lemma}
			\begin{proof}
				We've observed that the sequence $\left\{Y^n,K^n\right\}$ converges uniformly in $t$ in probability to $(Y, K)$, where $K=K^c+K^d$. From \eqref{equation satified by the limitting process Y : one reflected case}, it follows that for any $t \leq T$, $\Delta K^d_t=\sum_{0 < s \leq t} (Y_s-L_{s-})^{-}\mathds{1}_{\{\Delta L_s <0 \}}$. The sequence of measures $(dK^n_t)_{n \in \mathbb{N}}$ converges in total variation norm to $dK_t$, $\mathbb{P}$-a.s., for all $t \in [0,T]$. Specifically, we have the following convergences in the sense of measures:
				\begin{itemize}
					\item[($*$)] $K^d_t =\lim\limits_{n \rightarrow \infty} K^{n,d}_t=\lim\limits_{n \rightarrow \infty} n \sum_{0 < s \leq t} (Y^n_s-L_{s-})^{-} \mathds{1}_{\{\Delta L_s <0 \}}=\sum_{0 < s \leq t} (Y_s-L_{s-})^{-} \mathds{1}_{\{\Delta L_s <0 \}}$
					\item[($**$)] $ d K^{n,c}_t\rightarrow dK^{c}_t, \text{ weakly in } [0,T]$.
				\end{itemize}
				Using the convergence property of the stochastic integral in the \textsc{UCP} topology, we can easily deduce that
				\begin{equation*}
					\int_0^T L_sdK^{n,c}_s \xrightarrow[n \rightarrow \infty]{\mathbb{P}} \int_0^T L_sdK^{c}_s
					\text{ and }
					\int_0^T Y_sdK^{n,c}_s \xrightarrow[n \rightarrow \infty]{\mathbb{P}} \int_0^T Y_sdK^{c}_s.
					\label{ucp convergence}
				\end{equation*}
				Furthermore, from the Cauchy-Schwarz inequality, we obtain
				\begin{equation*}
					\begin{split}
						\mathbb{E}\left[\left| \int_0^T (Y^n_s-Y_s)dK^{n,c}_s \right|  \right]
						\leq  \left(\mathbb{E}\left[ \sup_{0 \leq s \leq T} e^{\mu A_s}\left| Y^n_s-Y_s \right|^2 \right]\right)^\frac{1}{2}
						\left(\mathbb{E}\left| K^{n}_T \right|^2   \right)^\frac{1}{2} \xrightarrow[n \rightarrow \infty]{} 0.
					\end{split}
				\end{equation*}
				Finally,
				\begin{equation*}
					\begin{split}
						&\int_0^T (Y^n_s-L_s)dK^{n,c}_s-\int_0^T (Y_s-L_s)dK^{c}_s\\
						&=\int_0^T (Y^n_s-Y_s) dK^{n,c}_s+\int_0^T Y_s dK^{n,c}_s-\int_0^T Y_s dK^c_s- \int_0^TL_sdK^{n,c}_s +\int_0^T L_s dK^c_s
						\xrightarrow[n \rightarrow \infty]{\mathbb{P}} 0.
					\end{split}
				\end{equation*}
				Hence, passing again to a subsequence, we get
				\begin{equation*}
					\begin{split}
						&\int_0^T (Y^n_s-L_s)dK^{n,c}_s\xrightarrow[n \rightarrow \infty]{\text{ a.s. }} \int_0^T (Y_s-L_s)dK^{c}_s
					\end{split}
				\end{equation*}
				Remark that $\int_0^T (Y^n_s-L_s)dK^{n,c}_s \leq 0$, $\forall n \in \mathbb{N}$,  which implies that  $\int_0^T (Y_s-L_s)dK^{c}_s \leq 0$. On the other hand, since $Y_t \geq L_t$ for any $t \leq T$, we deduce that $\int_0^T (Y_s-L_s)dK^{c}_s \geq 0$.\\
				We conclude that $\int_0^T (Y_t-L_t)dK^{c}_t=0$ which gives the desired result.
			\end{proof}
			To finalize the proof of Theorem \ref{th f ind zv}, we must establish the strong integrability condition for the quadruplet $(Y,Z,V,M)$.
			\subparagraph*{Step 7: The limiting process $(Y,Z,V,M)$ belongs to $\mathcal{S}^{2}_{\mu}(A;\mathbb{R}) \times M^2_{\mu}(\mathbb{R}^d) \times M^2_{\mu}(\mathbb{L}^2_Q)\times \mathbb{M}_{\mu}^2$}
			%In this section we will show that the
			Let now $(Y,Z,V,K,M)$ be the  solution of the GRBSDE associated with $(\xi,\mathfrak{f},g,A,L)$ constructed previously, which satisfies
			\begin{equation*}
				\begin{split}
					\mathbb{E}\left[\sup_{0 \leq t \leq T} \left| Y_t \right|^2 +\int_0^T  \left| Y_s \right|^2 dA_s+ \int_0^T \left\{ \left| Y_s \right|^2+\left\| Z_s \right\|^2 +\left\| V_s \right\|_Q^2  \right\} ds+[M]_T \right]<+\infty.
				\end{split}
			\end{equation*}
			The following lemma is useful in order to obtain the existence result in the general case when $f$ also depends on $(z,v)$ via a fixed point argument.
			\begin{lemma}
				For $\mu >1$, there exists a constant $\mathfrak{c}=\mathfrak{c}_{\mu,\alpha,\beta,\kappa}>0$ such that
				\begin{equation*}
					\begin{split}
						&\mathbb{E}\left[\sup_{0 \leq t \leq T} e^{\mu A_t}\left| Y_t \right|^2+ \int_0^T e^{\mu A_s}\left\{ \left| Y_s \right|^2dA_s+\left( \left\| Z_s \right\|^2 +\left\| V_s \right\|_Q^2 \right)ds  \right\} ds+\int_0^T e^{\mu A_s}d[M]_s +(K_T)^2\right]\\
						&\leq \mathfrak{c}\left\{ \mathbb{E}\left[e^{\mu A_T} \left| \xi \right|^2 \right] +\mathbb{E}\left[\int_0^T  e^{\mu A_s} \left(\varphi^2_s ds+\psi^2_s dA_s  \right) \right]+\mathbb{E}\left[\sup_{0 \leq t \leq T} \left| e^{\mu A_t} L_t^{+}\right|^2\right]   \right\}
					\end{split}
					\label{estimation for step 5}
				\end{equation*}
				\label{Step 5 : estimation lemma}
			\end{lemma}
			\begin{proof}
				Due to the lack of integrability of the processes $(Y,Z,V,M)$, we proceed by localization. Let $\{\tau_k\}_{k \geq 1}$ be a sequence of $\mathbb{F}$-stopping times defined as $\tau_k=\inf\{t \geq 0 : A_t \geq k \} \wedge T$ for $k \geq 1$. Note that $(\tau_k)_{k \geq 1}$ is a sequence of non-decreasing $\mathbb{F}$-stopping times of stationary type that converges to $T$ because, from Remark \ref{Rmq 1}, we have $\mathbb{P}$-a.s. $A_T < \infty$. Hence, it suffices to prove the result in the case where $A_T$ is bounded, and then apply Fatou's lemma (see  \cite[Proposition 1]{ELMandELOVMSTA} for a similar procedure).
				
				Next, by employing Itô's formula with the monotonicity and linear growth assumptions on the drivers $\mathfrak{f}$ and $g$, along with the Skorokhod condition \eqref{Skorohod condition}, and subsequently taking expectations, we can follow a procedure similar to that employed in Lemma \ref{Lemma of UI} to obtain the desired result.
				% derive the existence of a constant $\mathfrak{c} > 0$, such that
				%	\begin{equation}
					%		\begin{split}
						%			&\mathbb{E}\left[\sup_{0 \leq t \leq T}  e^{\mu A_t}\left| Y_t \right|^2 \right]  +\mathbb{E}\left[ \int_t^T e^{\mu A_s}\left| Y_s \right|^2 \left(ds+dA_s\right)\right] \\
						%			& +\mathbb{E}\left[\int_t^T e^{\mu A_s}\left\| Z_s \right\|^2 ds \right]+\mathbb{E}\left[\int_t^T e^{\mu A_s}\left\| V_s \right\|^2_Q ds \right]+\mathbb{E}\left[\int_t^T e^{\mu A_s}d[M]_s\right]\\
						%			& \leq \mathfrak{c} \left( \mathbb{E}\left[ e^{\mu A_T}\left| \xi \right|^2 \right]+ \mathbb{E}\left[ \int_t^T e^{\mu A_s} \varphi^2_s ds \right] + \mathbb{E}\left[ \int_t^T e^{\mu A_s} \psi^2_s dA_s \right]  +\mathbb{E}\left[\sup_{0 \leq t \leq T}\left| e^{\mu A_t} L^{+}_t \right|^2 \right]  \right)\\
						%		\end{split}
					%		\label{estimation for step 5}
					%	\end{equation}
				%	This completes the proof of Lemma \ref{Step 5 : estimation lemma}.
			\end{proof}
			
			The proof of Theorem \ref{th f ind zv} is now complete.
		\end{proof}

		Now, we turn to the general case of a generator $f$ depending on $(y,z,v)$.
		\paragraph*{Part 2: General case}
		In this part, the existence results will be obtained via a fixed point argument using the result from \textbf{Part 1}. Actually, let us set $\mathcal{B}^2_{\gamma,\mu}=M^2(\mathbb{R}) \times M^2_{\mu}(\mathbb{R}^d )\times M^2_{\mu}(\mathbb{L}^2_Q) \times \mathbb{M}_{\mu}^2$ endowed with the norm
		\begin{equation*}
			\|(Y,Z,V,M)  \|^2_{\gamma}= \mathbb{E}\left[ \int_0^T \Phi^{\gamma,\mu}_s \left( \left\{ \left| Y_s \right|^2+\| Z_s \|^2  +\| V_s \|^2_Q \right\} ds+ d[M]_s \right)  \right];\quad \gamma >0,\mu >1,
		\end{equation*}
		with $\Phi^{\gamma,\mu}_t=e^{\gamma t+\mu A_t}$, $t \in [0,T]$.
		
		Let $\Psi$ be the map from $\mathcal{B}^2_{\gamma,\mu}$ into itself defined by $\Psi(Y,Z,V,M)=(\bar{Y},\bar{Z},\bar{V},\bar{M})$ solution of the GRBSDE associated with parameters $(\xi,f(t,y,Z_t,V_t),g,A,L)$. Due to Theorem \ref{th f ind zv} this mapping is well defined.\\
		Let $(Y^{\prime},Z^{\prime},V^{\prime},M^{\prime})$ be another element of $\mathcal{B}^2_{\gamma,\mu}$ such that $\Psi(Y^{\prime},Z^{\prime},V^{\prime},M^{\prime})=(\bar{Y}^{\prime},\bar{Z}^{\prime},\bar{V}^{\prime},\bar{M}^{\prime})$.
		
		Applying It\^{o}'s formula and taking the expectation, it follows for $ \gamma > 0$
		\begin{equation}
			\begin{split}
				&\mathbb{E}\int_0^T \Phi^{\gamma,\mu}_s \left\{ \left( \gamma \left| \bar{Y}_s-\bar{Y}^{\prime}_s \right|^2  + \| \bar{Z}_s-\bar{Z}^{\prime}_s \|^2+\left\| \bar{V}_s-V^{\prime}_s \right\|^2_Q\right)ds+\left(\mu -\beta \right)  \left|  \bar{Y}_s-\bar{Y}^{\prime}_s \right|^2 dA_s+d\left[\bar{M}-\bar{M}^{\prime} \right]_s \right\} \\
				&\leq 2\mathbb{E}\left[ \int_0^T \Phi^{\gamma,\mu}_s(\bar{Y}_s-\bar{Y}^{\prime}_s)\left\lbrace (f(s,\bar{Y}_s,Z_s,V_s)-f(s,\bar{Y}^{\prime}_s,Z_s,V_s))+f(s,\bar{Y}^{\prime}_s,Z_s,V_s)-f(s,\bar{Y}^{\prime}_s,Z^{\prime}_s,V^{\prime}_s)\right\rbrace  ds \right]\\
				&\quad + 2\mathbb{E}\left[ \int_0^T \Phi^{\gamma,\mu}_s \left( \bar{Y}_s-\bar{Y}^{\prime}_s \right) \left( g(s,\bar{Y}_s)-g(s,\bar{Y}^{\prime}_s) \right) dA_s \right]\\
				& \leq \left( 2 \left| \alpha \right| +4 \kappa^2 \right) \mathbb{E}\left[ \int_0^T\Phi^{\gamma,\mu}_s \left| \bar{Y}_s-\bar{Y}^{\prime}_s \right|^2 ds \right]+\dfrac{1}{2}\mathbb{E}\left[ \int_0^T \Phi^{\gamma,\mu}_s \left\{ \left\| Z_s-Z^{\prime}_s \right\|^2 + \| V_s-V^{\prime}_s \|_Q^2 \right\} ds \right].
			\end{split}
			\label{Going back to this --after reflection inequality}
		\end{equation}
		This is due to the assumption \textbf{(H2)}-(iii)-(iv), the martingale property of the stochastic integrals using the Burkholder-Davis-Gundy's inequality and the fact that
		\begin{equation*}
			\int_{t }^{T} \Phi^{\gamma,\mu}_s(\bar{Y}_{s-}-\bar{Y}^{\prime}_{s-})(d\bar{K}_s-d\bar{K}'_s) \leq 0,\quad \forall t \leq T,
		\end{equation*}
		which stems from the Skorokhod conditions satisfied by the processes $(\bar{Y}_t)_{0 \leq t \leq T}$ and $(\bar{Y}^{\prime}_t)_{0 \leq t \leq T}$ under the actions of the two reflection processes $\bar{K}$ and $\bar{K}^{\prime}$ over $[t,T]$. Indeed, for any $t \in [0,T]$,
		\begin{equation*}
			\begin{split}
				&\int_{t}^{T} \Phi^{\gamma,\mu}_s(\bar{Y}_{s-}-\bar{Y}^{\prime}_{s-})(d\bar{K}_s-d\bar{K}'_s)=\int_{t}^{T} \Phi^{\gamma,\mu}_s(\bar{Y}_{s-}-\bar{Y}^{\prime}_{s-})(d\bar{K}^{c}_s-d\bar{K}^{\prime c}_s) +\int_{t}^{T} \Phi^{\gamma,\mu}_s(\bar{Y}_{s-}-\bar{Y}^{\prime}_{s-})(d\bar{K}^{d}_s-d\bar{K}^{\prime d}_s)
			\end{split}
		\end{equation*}
		The processes $(\bar{Y}_t)_{t \in [0,T]}$ and $(\bar{Y}^{\prime}_t)_{0 \leq t \leq T}$ are RCLL with non-negative jumps. Therefore, the jump sets $\delta_t(\omega)=\{s \in [t,T] : \Delta \bar{Y}_s\neq 0\}$ and $\delta_t^{\prime}(\omega)=\{s \in [t,T] : \Delta \bar{Y}'_s\neq 0\}$ are at most countable. It follows that
		\begin{equation*}
			\begin{split}
				\int_{t}^{T} \Phi^{\gamma,\mu}_s(\bar{Y}_{s-}-\bar{Y}^{\prime}_{s-})(d\bar{K}^{c}_s-d\bar{K}^{\prime c}_s) 
				&=-\int_{t}^{T} \Phi^{\gamma,\mu}_s(\bar{Y}_{s}-L_{s})d\bar{K}^{\prime c}_s+\int_{t}^{T} \Phi^{\gamma,\mu}_s (L_s-\bar{Y}^{\prime}_{s})d\bar{K}^{c}_s \leq 0.
			\end{split}
		\end{equation*}
		On the other hand,
		\begin{equation}
			\begin{split}
				\int_{t}^{T}\Phi^{\gamma,\mu}_s(\bar{Y}_{s-}-\bar{Y}^{\prime}_{s-})(d\bar{K}^{d}_s-d\bar{K}^{\prime d}_s)
				=\int_{t}^{T} \Phi^{\gamma,\mu}_s(\bar{Y}_{s-}-\bar{Y}^{\prime}_{s-})d\bar{K}^{d}_s-\int_{t}^{T}\Phi^{\gamma,\mu}_s(\bar{Y}_{s-}-\bar{Y}^{\prime}_{s-}) d\bar{K}^{\prime d}_s
			\end{split}
			\label{least or equal to zero 1}
		\end{equation}
		and
		\begin{equation}
			\begin{split}
				&\int_{t}^{T}\Phi^{\gamma,\mu}_s(\bar{Y}_{s-}-\bar{Y}^{\prime}_{s-}) d\bar{K}^{\prime d}_s
				=\int_{t}^{T} \Phi^{\gamma,\mu}_s(\bar{Y}_{s-}-\bar{Y}_{s}) d\bar{K}^{\prime d}_s
				+\int_{t}^{T} \Phi^{\gamma,\mu}_s(\bar{Y}_{s}-\bar{Y}^{\prime}_{s}) d\bar{K}^{\prime d}_s+\int_{t}^{T} \Phi^{\gamma,\mu}_s(\bar{Y}^{\prime}_{s}-\bar{Y}^{\prime}_{s-}) d\bar{K}^{\prime d}_s.
			\end{split}
			\label{least or equal to zero 2}
		\end{equation}
		From the definition of the predictable jumps of the processes $(\bar{Y}_t)_{0 \leq t \leq T}$ and $(\bar{Y}'_t)_{0 \leq t \leq T}$, the first and last terms on the right-hand side of \eqref{least or equal to zero 2} are equal to $\sum_{t< s \leq T} \Psi^{\gamma,\mu}_s \Delta \bar{K}^d_s \Delta \bar{K}^{\prime d}_s$ and $-\sum_{t< s \leq T} \Psi^{\gamma,\mu}_s \left(\Delta \bar{Y}^{\prime}_s \right)^2$, respectively. For the second term, we have
		\begin{equation*}
			\begin{split}
				&\int_{t}^{T}\Phi^{\gamma,\mu}_s(\bar{Y}_{s}-\bar{Y}^{\prime}_{s}) d\bar{K}^{\prime d}_s\\
				&=\int_{t}^{T} \Phi^{\gamma,\mu}_s(\bar{Y}_{s}-L_{s-}) d\bar{K}^{\prime d}_s+\int_{t}^{T} \Phi^{\gamma,\mu}_s(L_{s-}-\bar{Y}^{\prime}_{s}) d\bar{K}^{\prime d}_s\\
				&=\int_{t}^{T} \Phi^{\gamma,\mu}_s(\bar{Y}_{s}-L_{s-})\mathds{1}_{\{ \Delta \bar{Y}_{s} =0  \}} d\bar{K}^{\prime d}_s+\int_{t}^{T} \Phi^{\gamma,\mu}_s(\bar{Y}_{s}-L_{s-})\mathds{1}_{\{ \Delta \bar{Y}_{s} \neq 0  \}} d\bar{K}^{\prime d}_s
				+\int_{t}^{T} \Phi^{\gamma,\mu}_s(L_{s-}-\bar{Y}^{\prime}_{s}) d\bar{K}^{\prime d}_s\\
				& \geq -\sum_{t < s \leq T}\Phi^{\gamma,\mu}_s \Delta \bar{K}^d_s \Delta \bar{K}^{\prime d}_s+\sum_{t < s \leq T}\Phi^{\gamma,\mu}_s \left(\Delta \bar{Y}^{\prime}_s \right)^2.
			\end{split}
		\end{equation*}
		Plugging this into \eqref{least or equal to zero 1}, we deduce that $\int_{t}^{T} \Phi^{\gamma,\mu}_s(\bar{Y}_{s-}-\bar{Y}^{\prime}_{s-})(d\bar{K}^{d}_s-d\bar{K}^{\prime d}_s) \leq 0$ for any $t \in [0,T]$. Going back to \eqref{Going back to this --after reflection inequality}, choosing $\gamma = 1+2\left| \alpha \right| +4 \kappa^2$ and $\mu \geq 1$, we obtain
		\begin{equation*}
			\begin{split}
				&\mathbb{E}\left[\int_0^T \Phi^{\gamma,\mu}_s\left( \left\{  \left| \bar{Y}_s-\bar{Y}^{\prime}_s \right|^2 + \| \bar{Z}_s-\bar{Z}^{\prime}_s \|^2+\| \bar{V}_s-\bar{V}^{\prime}_s \|_Q^2 \right\}ds + \left|  \bar{Y}_s-\bar{Y}^{\prime}_s \right|^2 dA_s +d\left[\bar{M}-\bar{M}^{\prime} \right]_s \right)\right]\\
				& \leq \dfrac{1}{2}\mathbb{E}\left[ \int_0^T \Phi^{\gamma,\mu}_s\left( \left\{ \left| Y_s-Y^{\prime}_s \right|^2+\left\| Z_s-Z^{\prime}_s \right\|^2 + \| V_s-V^{\prime}_s \|_Q^2 \right\} ds+\left| Y_s-Y^{\prime}_s \right|^2 dA_s+d\left[ M-M^{\prime}\right]_s \right)\right].
			\end{split}
		\end{equation*}
		Henceforth
		\begin{equation*}
			\left\| \Psi(Y,Z,V,M)-\Psi(Y^{\prime},Z^{\prime},V^{\prime},M^{\prime}) \right\|^2_{\gamma,\mu} \leq \dfrac{1}{2} \left\| Y-Y^{\prime},Z-Z^{\prime},V-V^{\prime},M-M^{\prime}  \right\|^2_{\gamma,\mu}
		\end{equation*}
		
		Finally, $\Psi$ is a strict contraction on the Banach space $\mathcal{B}^2_{\mu,\gamma}$ equipped with the norm $\left\| \cdot \right\|_{\gamma,\mu}$ and by Banach fixed point theorem we conclude that $\Psi$ has a unique fixed point $(Y,Z,V,M)$ which, with $K$ solves the GRBSDE \eqref{basic equation for one lower reflecting barrier--General case}.
	\end{proof}
	\subsection{Generalized BSDEs with one upper reflecting barrier}
	\textbf{(H3")} Let $(U_t)_{0 \leq t \leq T}$ be an $\mathcal{F}_t$-progressively measurable RCLL real-valued process satisfying:
	\begin{equation*}
		\mathbb{E}\left[ \sup_{0 \leq t \leq T} \left| e^{\mu A_t} U_t^{-} \right|^2 \right]<\infty,\quad \text{ and } \quad \xi \leq U_T,~\text{ a.s.}
	\end{equation*}
	
	Following a similar argument as the one used in the proof of Theorem \ref{Theorem : GRBSDE with one lower barrier--Main results}, we may show the following result:
	\begin{theorem}
		Assume that conditions \textbf{(H1)}, \textbf{(H2)} and \textbf{(H3")} holds. Then, the GRBSDE with one upper reflecting barrier $U$ associated with $(\xi,f,g,A,U)$:
		\begin{equation}
			\left\{
			\begin{split}
				\text{(i)}&~Y_t= \xi+\int_t^T f(s,Y_s,Z_s,V_s)ds+\int_t^T g(s,Y_s)\mathrm{d}A_s-(K_T-K_t)\\
				&\qquad\qquad-\int_t^T Z_s d B_s -\int_t^T \int_E V_s(e)\tilde{N}(ds,de)-\int_t^T dM_s,\quad 0 \leq t \leq T.\\
				\text{(ii)} &~ Y \leq U, \text{ and if } K^{c} \text{ is the continuous part of } K, \text{ then }
				\int_{0}^{T}(U_t-Y_t)dK^{c}_t=0.\\
				\text{(iii)} &~ \text{ If } K^{d} \text{ is the purely discontinuous part of } K,
				\text{ then } K^{d}_t=\sum_{0 < s \leq t}(Y_s-U_{s-})^{+}\mathds{1}_{\{U_{s-}=Y_{s-}\}},
			\end{split}
			\right.
			\label{basic equation for one upper reflecting barrier}
		\end{equation}
		admit a unique solution $(Y,Z,V,K,M) \in \mathcal{S}^{2}_{\mu}(A;\mathbb{R}) \times M^2_{\mu}(\mathbb{R}^d) \times M^2_{\mu}(\mathbb{L}^2_Q) \times \mathcal{S}^2 \times \mathbb{M}_{\mu}^2$. Moreover,  there exists a constant $\mathfrak{c}>0$ such that
		\begin{equation*}
			\begin{split}
				& \mathbb{E}\left[   \sup_{0 \leq t \leq T}  e^{\mu A_t} \left| Y_t \right|^2+  \int_0^T e^{\mu A_s}  \left| Y_s \right|^2 dA_s+  \int_0^T e^{\mu A_s}  \left\lbrace  \| Z_s\|^2+\left\| V_s \right\|^2_Q \right\rbrace  ds + \int_0^T e^{\mu A_s}  d[M ]_s + \left( K_T \right)^2 \right]\\
				&\leq  \mathfrak{c} \mathbb{E}\left[  e^{\mu A_T}\left| \xi \right|^2 + \int_0^T e^{\mu A_s}  \varphi_s^2 ds
				+\int_0^T e^{\mu A_s}  \psi^2_s dA_s + \sup_{0 \leq t \leq T} e^{2\mu A_t} (U_t^{-})^2  \right].
			\end{split}
			\label{Estimation for the upper barrier : main result}
		\end{equation*}
		
		\label{Corllary of the GRBSDE with one reflecting upper barrier}
	\end{theorem}

		\subsection{Comparison principal}
	In this part, we provide some classical comparison theorems that will be used toward the rest of the paper.
	\begin{theorem}\label{Version comparison theorem -- independ of v}
		Assume that:
		\begin{itemize}
			\item[(i)] $f$ is independent of $v$ ,
			\item[(ii)] For any $t \leq T$, $f(t,Y^{\prime}_t,Z^{\prime}_t) \leq f^{\prime}(t,Y^{\prime}_t,Z^{\prime}_t,V^{\prime}_t)$, $g(t,Y^{\prime}_t) \leq g^{\prime}(t,Y^{\prime}_t)$ and $\xi \leq \xi^{\prime}$.
		\end{itemize}
		Then  $Y_t \leq Y^{\prime}_t$, $\forall t \leq T$ a.s.
	\end{theorem}
	\begin{proof}
		We set : $\bar{Z}=Z-Z^{\prime}$, $\bar{V}=V-V^{\prime}$, $\bar{M}=M-M^{\prime}$ and $\psi(x)=(x^{+})^2$, $x \in \mathbb{R}$.
		The main tool is to make us use of Meyer-It\^{o}'s formula (see \cite{Protter}, Theorem IV.70) to $\psi(Y-Y^{\prime})$ between $t$ and $T$, which gives
		\begin{equation*}
			\begin{split}
				&\psi(Y_t-Y^{\prime}_t)+\int_t^T \mathds{1}_{ \{Y_{s-}>Y^{\prime}_{s-} \} } \left\| \bar{Z}_s \right\|^2 ds
				+\int_t^T \mathds{1}_{ \{Y_{s-}>Y^{\prime}_{s-} \} }  d \left\langle \bar{M}^c\right\rangle_s \\
				&\leq \psi(\xi-\xi^{\prime})+ \int_t^T  \psi^{\prime}(Y_{s-}-Y^{\prime}_{s-}) \left(f(s,Y_s,Z_s)-f^{\prime}(s,Y^{\prime}_s,Z^{\prime}_s,V^{\prime}_s)\right)ds\\
				&\quad+\int_t^T  \psi^{\prime}(Y_{s-}-Y^{\prime}_{s-})(g(s,Y_s)-g^{\prime}(s,Y^{\prime}_s))dA_s
				-2 \int_t^T  \mathds{1}_{ \{Y_{s-}>Y^{\prime}_{s-} \} } (Y_{s-}-Y^{\prime}_{s-})^{+}\bar{Z}_s dB_s\\
				&\quad -2 \int_t^T \int_E  \mathds{1}_{ \{Y_{s-}>Y^{\prime}_{s-} \} } (Y_{s-}-Y^{\prime}_{s-})^{+} \bar{V}_s(e)\tilde{N}(ds,de)
				-2 \int_t^T  \mathds{1}_{ \{Y_{s-}>Y^{\prime}_{s-} \} } (Y_{s-}-Y^{\prime}_{s-})^{+} d\bar{M}_s,
			\end{split}
		\end{equation*}
		since $\int_t^T \psi^{\prime}(Y_{s-}-Y^{\prime}_{s-}) (dK_s - dK^{\prime}_s) \leq 0$ and
		$
		\sum_{t < s \leq T} \left\{  \psi(Y_s-Y^{\prime}_s)-\psi(Y_{s-}-Y^{\prime}_{s-}) - \psi^{\prime}(Y_{s-}-Y^{\prime}_{s-})\Delta\left( 	Y-Y^{\prime}\right)_s \right\} \geq 0.
		$
		%\newpage
		Taking expectation in both hand-sides yields
		%\begin{equation*}
		%\begin{split}
		%&\mathbb{E}\left[ \left| (Y_t-Y^{\prime}_t)^{+} \right|^2\right]+\mathbb{E}\int_t^T \mathds{1}_{ \{Y_{s-}>Y^{\prime}_{s-} \} } \left| Z_s -Z^{\prime}_s \right|^2 ds \\
		%&\leq 2 \mathbb{E}\int_t^T  (Y_{s-}-Y^{\prime}_{s-})^{+} \left(f(s,Y_s,Z_s)-f(s,Y^{\prime}_s,Z^{\prime}_s)\right)ds\\
		%&\qquad+2 \mathbb{E}\int_t^T  (Y_{s-}-Y^{\prime}_{s-})^{+} \left(g(s,Y_s)-g(s,Y^{\prime}_s)\right)dA_s\\
		%&\leq \left( 2 \alpha^{+}+\kappa^2 \right) \int_t^T \mathbb{E}\left[ \left| (Y_{s}-Y^{\prime}_{s})^{+}\right|^2\right] ds+ \mathbb{E}\int_t^T  \mathds{1}_{\{ Y_{s-}-Y^{\prime}_{s-}>0\}}  \left| Z_s-Z_{s}^{\prime} \right|^2 ds.
		%\end{split}
		%\end{equation*}
		\begin{equation*}
			\begin{split}
				&\mathbb{E}\left[ \left| (Y_t-Y^{\prime}_t)^{+} \right|^2\right]+\mathbb{E}\int_t^T \mathds{1}_{ \{Y_{s-}>Y^{\prime}_{s-} \} } \left\| Z_s -Z^{\prime}_s \right\|^2 ds \\
				&\leq \left( 2 \alpha^{+}+\kappa^2 \right) \int_t^T \mathbb{E}\left[ \left| (Y_{s}-Y^{\prime}_{s})^{+}\right|^2\right] ds+ \mathbb{E}\int_t^T  \mathds{1}_{\{ Y_{s-}-Y^{\prime}_{s-}>0\}}  \left\| Z_s-Z_{s}^{\prime} \right\|^2 ds.
			\end{split}
		\end{equation*}
		From Gronwall's lemma, we derive $(Y_t-Y^{\prime}_t)^{+} =0$, for each $t \leq T$. Then the result follows from the RCLL property of the paths of the processes $Y$ and $Y^{\prime}$.
		%	The proof is then complete.
		
	\end{proof}
	%%%
	\begin{remark}\label{Remarks on the version 1 of comparison result}
		\begin{enumerate}
			\item The same comparison result holds for GRBSDE \eqref{basic equation for one upper reflecting barrier} with one upper reflecting barrier.
			\item From the definition of the predictable jumps of the two RCLL processes $Y$ and $Y'$, if $Y \leq Y'$, then we obviously have $\mathbb{P}$-a.s., for any $s \leq t$, $K^d_t - K^d_s \geq K^{\prime,d}_t - K^{\prime,d}_s$. Additionally, if $f$ is also independent of $z$ and $f'$ does not depend on $(z,v)$, then we have $K_t - K_s \geq K^{\prime}_t - K^{\prime}_s$, for any $0 \leq s \leq t \leq T$.
			\item The comparison Theorem \ref{Version comparison theorem -- independ of v} remains true even if $f$ depends also on $v$. But in this case, we have to assume additionally that: For each $(y,z,v,v') \in \mathbb{R} \times \mathbb{R}^d \times \mathbb{L}^2_Q \times \mathbb{L}^2_Q$, there exists a $\mathcal{P} \otimes \mathcal{E}$-measurable process $\gamma^{y,z,v,v'} : \Omega \times [0,T] \times E \rightarrow \mathbb{R}$ such that
			\begin{equation*}
				f(s,y,z,v) - f(s,y,z,v') \leq \int_E \left( v(e) - v'(e) \right) \gamma^{y,z,v,v'}_s(e)Q(s,de)\eta(s),
				\label{Girsanov--Theorem application}
			\end{equation*}
			with $d\mathbb{P} \otimes ds \otimes Q(s,de)\eta(s)$-a.e. for any $(y,z,v,v')$,
			\begin{itemize}
				\item[$\ast$] $-1 \leq \gamma^{y,z,v,v'}_s(e)$,
				\item[$\ast \ast$] $\left| \gamma^{y,z,v,v'}_s(e) \right| \leq \nu(e)$ where $\nu \in \mathbb{L}^2_Q$.
			\end{itemize}
		\end{enumerate}
	\end{remark}
%	\section{Applications}
\section{Link with optimal stopping problems}
\label{app}
	We provide the characterization of the state process $(Y_t)_{0 \leq t \leq T}$ of GRBSDEs given by \eqref{basic equation for one lower reflecting barrier--General case} or \eqref{basic equation for one upper reflecting barrier} as the value function of associated optimal stopping problems.
	
	\begin{proposition}\label{Remark : Essup in the case where f is independent of z and v}
		Assume that conditions (\textbf{H1}), (\textbf{H2}),  (\textbf{H3}) hold. 
		\begin{itemize}
			\item[$\bullet$] Let $(\underline{Y}_t,\underline{Z}_t,\underline{V}_t,\underline{K}_t,\underline{M}_t)_{0 \leq t \leq T}$ be the unique solution of the GRBSDE \eqref{basic equation for one lower reflecting barrier--General case}. Then, we have
			\begin{equation}
				\begin{split}
					&\underline{Y}_t = \esssup_{\tau \in \mathcal{T}_t^T}\mathbb{E}\left[ L_{\tau}\mathds{1}_{\{ \tau < T\}}+\xi \mathds{1}_{\{ \tau = T\}} +\int_t^{\tau} f(s,\underline{Y}_s,\underline{Z}_s,\underline{V}_s)ds+\int_t^{\tau} g(s,\underline{Y}_s) dA_s \mid \mathcal{F}_t \right].
				\end{split}
				\label{state process--ess sup--lower barrier}
			\end{equation}
			
			\item[$\bullet$] Let $(\bar{Y}_t,\bar{Z}_t,\bar{V}_t,\bar{K}_t,\bar{M}_t)_{0 \leq t \leq T}$ be the unique solution of the GRBSDE \eqref{basic equation for one upper reflecting barrier}. Then, we have
			\begin{equation*}
				\begin{split}
					&\bar{Y}_t = \essinf_{\tau \in \mathcal{T}_t^T}\mathbb{E}\left[ U_{\tau}\mathds{1}_{\{ \tau < T\}}+\xi \mathds{1}_{\{ \tau = T\}} +\int_t^{\tau} f(s,\bar{Y}_s,\bar{Z}_s,\bar{V}_s)ds+\int_t^{\tau} g(s,\bar{Y}_s) dA_s \mid \mathcal{F}_t \right].
				\end{split}
			\end{equation*}
		\end{itemize}
	\end{proposition}
	\begin{proof}
		We provide the proof for the one lower reflected generalized BSDE, as the argument for the one upper case is similar.
		
		To simplify notation, we denote by $\left(Y_t,Z_t,V_t,K_t,M_t\right)_{0 \leq t \leq T}$ the unique solution of the GRBSDE \eqref{basic equation for one lower reflecting barrier--General case}. In other words, $\left(Y_t,Z_t,V_t,K_t,M_t\right)_{0 \leq t \leq T}=(\underline{Y}_t,\underline{Z}_t,\underline{V}_t,\underline{K}_t,\underline{M}_t)_{0 \leq t \leq T}$, which exists due to Theorem \ref{Theorem : GRBSDE with one lower barrier--Main results}.

		Note that $\left( Y_t+\int_0^t f(s,Y_s,Z_s,V_s)ds+\int_0^t g(s,Y_s)dA_s   \right)_{0 \leq t \leq T}$ is an RCLL supermartingale. It follows from the properties of the Snell envelope that
		\begin{equation}
			Y_t \geq \esssup_{\tau \in \mathcal{T}_t^T} \mathbb{E}\left[ \int_t^{\tau} f(s,Y_s,Z_s,V_s)ds+\int_t^{\tau }g(s,Y_s)dA_s+L_{\tau} \mathds{1}_{\{ \tau<T\}} +\xi \mathds{1}_{\{\tau=T\}} \mid \mathcal{F}_t \right].
			\label{this combined with this}
		\end{equation}
		To obtain the reversible inequality, we now pick a suitable sequence of stopping times in $\mathcal{T}_t^T$:\\ Let $t \in [0,T]$ be fixed. For $p \geq 1$, define the sequence $\{\nu^p_t\}_{p \geq 1}$ as follows:
		\begin{equation*}
			\nu^p_t=\inf\{s \geq t : Y_s=L_s \text{ or } Y_s \leq L_s +\frac{1}{p}\} \wedge T.
		\end{equation*}
		From its definition, it's clear that $Y_{s-}>L_{s-}$ on $[t,\nu^p_t]$ and $Y_s>L_s$ on $[t,\nu^p_t[$. Now from the Skorokhod condition $\int_0^T (Y_{s-}-L_{s-})dK_s=0$, the definition of the jump process $(K^{d}_t)_{0 \leq t \leq T}$ and the continuity of $(K^c_t)_{0 \leq t \leq T}$ we deduce that: $\int_t^{\nu^p_t} dK^{d}_s=\int_t^{\nu^p_t} dK^{c}_s=0$, hence $\int_t^{\nu^p_t} dK_s=0$. Moreover, it holds that $Y_s \mathds{1}_{\{ \nu^p_t < T \}}=L_s \mathds{1}_{\{ \nu^p_t < T \}}$ or $Y_s \mathds{1}_{\{ \nu^p_t < T \}} \leq L_s \mathds{1}_{\{ \nu^p_t < T \}} +\frac{1}{p}$. Indeed, let $\omega \in \Omega$ be fixed such that $\nu^p_t(\omega)<T$. Then, there exists a sequence $(t_k)_{k \in \mathbb{N}}$ of real numbers which decrease to $\nu^p_t(\omega)$ such that $Y_{t_k} = L_{t_{k}}$ or $Y_{t_k} \leq L_{t_k}+\frac{1}{p}$ for all $k \in \mathbb{N}$. Since the obstacle $L$ and the state process $Y$ have right continuous trajectories, thus taking the limit as $k \rightarrow \infty$ we get: $Y_{\nu^p_t}=L_{\nu^p_t}$ or $Y_{\nu^p_t} \leq L_{\nu^p_t}+\frac{1}{p}$, which implies $Y_{\nu^p_t}\mathds{1}_{\{ \nu^p_t < T \}}=L_{\nu^p_t}\mathds{1}_{\{ \nu^p_t < T \}}$ or $Y_{\nu^p_t} \mathds{1}_{\{ \nu^p_t < T \}} \leq  L_{\nu^p_t} \mathds{1}_{\{ \nu^p_t < T \}}+  \frac{1}{p} $. Thus, in both cases, we conclude that $Y_{\nu^p_t} \mathds{1}_{\{ \nu^p_t < T \}} \leq  L_{\nu^p_t} \mathds{1}_{\{ \nu^p_t < T \}}+  \frac{1}{p} $ holds. But from the GRBSDE \eqref{basic equation for one lower reflecting barrier--General case}, we have
		\begin{equation*}
			\begin{split}
				Y_{t}&= Y_{\nu^p_t}+\int_t^{\nu^p_t} f(s,Y_s,Z_s,V_s)ds+\int_t^{\nu^p_t} g(s,Y_s)\mathrm{d}A_s-\int_t^{\nu^p_t} Z_s d B_s   -\int_t^{\nu^p_t} \int_E V_s(e)\tilde{N}(ds,de)-\int_t^{\nu^p_t} dM_s\\
				&=Y_{\nu^p_t}\mathds{1}_{\{ \nu^p_t < T \}}+\xi \mathds{1}_{\{ \nu^p_t = T \}}+\int_t^{\nu^p_t} f(s,Y_s,Z_s,V_s)ds+\int_t^{\nu^p_t} g(s,Y_s)\mathrm{d}A_s-\int_t^{\nu^p_t} Z_s d B_s\\
				& \qquad   -\int_t^{\nu^p_t} \int_E V_s(e)\tilde{N}(ds,de)-\int_t^{\nu^p_t} dM_s\\
				& \leq L_{\nu^p_t}\mathds{1}_{\{ \nu^p_t < T \}}+\xi \mathds{1}_{\{ \nu^p_t = T \}}+\int_t^{\nu^p_t} f(s,Y_s,Z_s,V_s)ds+\int_t^{\nu^p_t} g(s,Y_s)\mathrm{d}A_s\\
				& \qquad  -\int_t^{\nu^p_t} Z_s d B_s -\int_t^{\nu^p_t} \int_E V_s(e)\tilde{N}(ds,de)-\int_t^{\nu^p_t} dM_s+\frac{1}{p}.
			\end{split}
		\end{equation*}
		Taking the conditional expectation with respect to $\mathcal{F}_t$, we obtain for every $p \geq 1$,
		\begin{equation*}
			\begin{split}
				Y_{t}\leq \mathbb{E}\left[ \int_t^{\nu^p_t} f(s,Y_s,Z_s,V_s)ds+\int_t^{\nu^p_t} g(s,Y_s)\mathrm{d}A_s  +L_{\nu^p_t}\mathds{1}_{\{ \nu^p_t < T \}}+\xi \mathds{1}_{\{ \nu^p_t = T \}} \mid \mathcal{F}_t \right]+\frac{1}{p}.
			\end{split}
		\end{equation*}
		Then,
		\begin{equation*}
			\begin{split}
				Y_{t} & \leq \esssup_{\tau \in \mathcal{T}_t^T} \mathbb{E}\left[ \int_t^{\tau} f(s,Y_s,Z_s,V_s)ds+\int_t^{\tau }g(s,Y_s)dA_s+L_{\tau} \mathds{1}_{\{ \tau<T\}} +\xi \mathds{1}_{\{\tau=T\}} \mid \mathcal{F}_t \right]+\frac{1}{p},\\
			\end{split}
		\end{equation*}
		for an arbitrary $p \geq 1$. Next, letting $p \rightarrow \infty$ and using \eqref{this combined with this} we deduce that $(Y_t)_{0 \leq t \leq T}$ satisfies \eqref{state process--ess sup--lower barrier}.
	\end{proof}
	%%
	%%
	%
	%

	%
	%
%	Now, we state a special comparison theorem for the case where the generator $f$ depends only on the solution, which will be needed later. The proof of this result is standard and thus omitted.
%	\begin{theorem}
%		Suppose that $\xi^1, \xi^2$ satisfies \textbf{(H1)}, that $f^1$, $f^2$, $g^1$, and $g^2$ satisfy assumption \textbf{(H2)} , and that $K^1$ and $K^2$ are two  finite variational RCLL, adapted, processes with square integrable total variation. If there exist a triplet $\left(Y_t^i, Z_t^i,V^i_t,M^i_t\right)_{0 \leq t \leq T}, i=1,2$, that belongs to $\mathcal{S}^{2}_{\mu}(A;\mathbb{R}) \times M^2_{\mu}(\mathbb{R}^d) \times M^2_{\mu}(\mathbb{L}^2_Q) \times \mathbb{M}_{\mu}^2$ satisfying the equations
%		\begin{equation*}
%			Y_t^i=\xi^i+\int_t^T f^i\left(s, Y_s^i\right) d s+K_T^i-K_t^i-\int_t^T Z_s^i d B_s-\int_{t}^{T}\int_{E}V^i_s(e)\tilde{N}(ds,de)-\int_{t}^{T} dM^i _s, ~ i=1,2,
%			\label{BSDE Comparison}
%		\end{equation*}
%		and, moreover for any $0 \leq t \leq T$,
%		$$
%		f^1\left(t, Y_t^2\right) \geq f^2\left(t, Y_t^2\right),~d\mathbb{P}\otimes dt\text{-a.e.,}\quad g^1\left(t, Y_t^2\right) \geq g^2\left(t, Y_t^2\right),~d\mathbb{P}\otimes dA_t\text{-a.e.,}  \quad \xi^1 \geq \xi^2,
%		$$
%		and $K^1-K^2$ is an increasing process, then $Y_t^1 \geq Y_t^2$ for all $t \in [0,T]$ a.s.
%		\label{Comparison result for reflected BSDE}
%	\end{theorem}
	%
	%

	%%\--------------------------------------------------------------------------------------------\%
	%%\--------------------------------------------------------------------------------------------\%

\end{document}